\pdfoutput=1
\documentclass[letterpaper, oneside, reqno]{amsart}
\usepackage[margin=1in]{geometry}

\usepackage{amsmath,amssymb,mathtools}
\usepackage[foot]{amsaddr}

\usepackage{placeins}

\usepackage{amssymb} 
\usepackage{stmaryrd}
\usepackage{graphicx}
\usepackage[colorlinks=true, pdfstartview=FitV, linkcolor=blue, citecolor=Green, urlcolor=WildStrawberry, linktoc=page]{hyperref} 

\usepackage{array}
\usepackage{ragged2e}

\usepackage{bm}

\usepackage{dsfont}
\usepackage[T1]{fontenc}
\usepackage[utf8]{inputenc}

\DeclareFontFamily{OT1}{rsfs}{}
\DeclareFontShape{OT1}{rsfs}{n}{it}{<-> rsfs10}{}
\DeclareMathAlphabet{\mathscr}{OT1}{rsfs}{n}{it}
\usepackage{mathrsfs}
\usepackage{MnSymbol}
\usepackage{setspace}
\usepackage[nopatch=footnote]{microtype}
\usepackage{enumitem}
\usepackage[dvipsnames]{xcolor}

\usepackage[normalem]{ulem}

\usepackage{booktabs} 

\usepackage{comment}

\usepackage{braket}

\usepackage{etoolbox}
\newtoggle{focs}
\toggletrue{focs}
\newcommand{\iffocs}[2]{\iftoggle{focs}{#1}{#2}}

\definecolor{darkgreen}{rgb}{0,0.5,0}
\definecolor{darkblue}{rgb}{0,0,0.7}
\definecolor{darkred}{rgb}{0.9,0.1,0.1}

\newlength{\bibitemsep}
\let\oldthebibliography\thebibliography
\renewcommand\thebibliography[1]{%
  \oldthebibliography{#1}%
  \setlength{\parskip}{\bibitemsep}%
  \setlength{\itemsep}{-7pt}%
}

\newtheoremstyle{break}%
{}{}%
{\itshape}{}%
{\bfseries}{.\vphantom{$p_{p_{p_p}}$}}%
{\newline}
{\thmname{#1}\thmnumber{ #2}\thmnote{\ \,\textmd{(#3)}}}

\theoremstyle{break}

\newtheorem{proposition}{Proposition}
\newtheorem{theorem}[proposition]{Theorem}
\newtheorem{lemma}[proposition]{Lemma}
\newtheorem{corollary}[proposition]{Corollary}

\theoremstyle{remark}
\newtheorem{remark}[proposition]{Remark}

\theoremstyle{definition}
\newtheorem{definition}[proposition]{Definition}

\newtheorem*{lemma*}{Lemma}

\newcommand{\vocab}[1]{\emph{#1}}

\numberwithin{equation}{section}
\numberwithin{proposition}{section}
\numberwithin{figure}{section}
\numberwithin{table}{section}

\newcommand{\Z}{\mathbb{Z}}

\newcommand{\R}{\mathbb{R}}

\newcommand{\calN}{\mathcal N}

\renewcommand{\P}{\mathop{{}\mathbb{P}}}

\newcommand{\Cov}{\mathop{{}\boldsymbol{\mathrm{Cov}}}}
\newcommand{\E}{\mathop{{}\mathbb{E}}}

\renewcommand{\le}{\leqslant}
\renewcommand{\ge}{\geqslant}
\renewcommand{\leq}{\leqslant}
\renewcommand{\geq}{\geqslant}

\renewcommand{\subset}{\subseteq}
\renewcommand{\bar}{\overline}

\newcommand{\td}{\widetilde}

\renewcommand{\hat}{\widehat}

\newcommand{\1}{\mathbf{1}}

\newcommand{\eps}{\varepsilon}

\newenvironment{e*}{\begin{equation*}}{\end{equation*}\ignorespacesafterend}
\newcommand{\norm}[1]{\left\lVert{#1}\right\rVert}

\newcommand{\an}[1]{\left\langle#1\right\rangle}

\newcommand{\ot}{\otimes}

\DeclareMathOperator{\sech}{sech}

\newcommand{\Id}[0]{I}

\renewcommand{\norm}[1]{\left\lVert{#1}\right\rVert}

\newcommand{\Var}{\mathop{{}\boldsymbol{\mathrm{Var}}}}

\renewcommand{\le}{\leqslant}
\renewcommand{\leq}{\leqslant}
\renewcommand{\ge}{\geqslant}
\renewcommand{\geq}{\geqslant}

\usepackage{prettyref}
\newcommand{\savehyperref}[2]{\texorpdfstring{\hyperref[#1]{#2}}{#2}}

\protected\def\verythinspace{%
  \ifmmode
    \mskip0.5\thinmuskip
  \else
    \ifhmode
      \kern0.083em
    \fi
  \fi
}
\newrefformat{eq}{\savehyperref{#1}{\textup{(\ref*{#1})}}}
\newrefformat{e}{\savehyperref{#1}{\textup{(\ref*{#1})}}}
\newrefformat{ineq}{\savehyperref{#1}{\textup{(\ref*{#1})}}}
\newrefformat{eqn}{\savehyperref{#1}{\textup{(\ref*{#1})}}}
\newrefformat{l}{\savehyperref{#1}{Lemma~\ref*{#1}}}
\newrefformat{lem}{\savehyperref{#1}{Lemma~\ref*{#1}}}
\newrefformat{def}{\savehyperref{#1}{Definition~\ref*{#1}}}
\newrefformat{d}{\savehyperref{#1}{Definition~\ref*{#1}}}
\newrefformat{t}{\savehyperref{#1}{Theorem~\ref*{#1}}}
\newrefformat{thm}{\savehyperref{#1}{Theorem~\ref*{#1}}}
\newrefformat{cor}{\savehyperref{#1}{Corollary~\ref*{#1}}}
\newrefformat{c}{\savehyperref{#1}{Corollary~\ref*{#1}}}
\newrefformat{cha}{\savehyperref{#1}{Chapter~\ref*{#1}}}
\newrefformat{sec}{\savehyperref{#1}{\S\verythinspace\ref*{#1}}}
\newrefformat{s}{\savehyperref{#1}{\S\verythinspace\ref*{#1}}}
\newrefformat{subsec}{\savehyperref{#1}{\S\verythinspace\ref*{#1}}}
\newrefformat{app}{\savehyperref{#1}{\S\verythinspace\ref*{#1}}}
\newrefformat{tab}{\savehyperref{#1}{Table~\ref*{#1}}}
\newrefformat{fig}{\savehyperref{#1}{Figure~\ref*{#1}}}
\newrefformat{hyp}{\savehyperref{#1}{Hypothesis~\ref*{#1}}}
\newrefformat{alg}{\savehyperref{#1}{Algorithm~\ref*{#1}}}
\newrefformat{a}{\savehyperref{#1}{Algorithm~\ref*{#1}}}
\newrefformat{rem}{\savehyperref{#1}{Remark~\ref*{#1}}}
\newrefformat{item}{\savehyperref{#1}{Item~\ref*{#1}}}
\newrefformat{step}{\savehyperref{#1}{step~\ref*{#1}}}
\newrefformat{conj}{\savehyperref{#1}{Conjecture~\ref*{#1}}}
\newrefformat{fact}{\savehyperref{#1}{Fact~\ref*{#1}}}
\newrefformat{p}{\savehyperref{#1}{Proposition~\ref*{#1}}}
\newrefformat{prop}{\savehyperref{#1}{Proposition~\ref*{#1}}}
\newrefformat{prob}{\savehyperref{#1}{Problem~\ref*{#1}}}
\newrefformat{claim}{\savehyperref{#1}{Claim~\ref*{#1}}}
\newrefformat{clm}{\savehyperref{#1}{Claim~\ref*{#1}}}
\newrefformat{relax}{\savehyperref{#1}{Relaxation~\ref*{#1}}}
\newrefformat{rem}{\savehyperref{#1}{Remark~\ref*{#1}}}
\newrefformat{red}{\savehyperref{#1}{Reduction~\ref*{#1}}}
\newrefformat{part}{\savehyperref{#1}{Part~\ref*{#1}}}
\newrefformat{ex}{\savehyperref{#1}{Exercise~\ref*{#1}}}
\newrefformat{property}{\savehyperref{#1}{Property~\ref*{#1}}}
\newrefformat{type}{\savehyperref{#1}{Type~\ref*{#1}}}
\newrefformat{eg}{\savehyperref{#1}{Example~\ref*{#1}}}
\newrefformat{obs}{\savehyperref{#1}{Observation~\ref*{#1}}}
\newrefformat{que}{\savehyperref{#1}{Question~\ref*{#1}}}
\newrefformat{cond}{\savehyperref{#1}{Condition~\ref*{#1}}}
\newrefformat{ass}{\savehyperref{#1}{Assumption~\ref*{#1}}}
\newrefformat{not}{\savehyperref{#1}{Notation~\ref*{#1}}}
\newrefformat{cond}{\savehyperref{#1}{Condition~\ref*{#1}}}
\newcommand{\Sref}[1]{\hyperref[#1]{\S\ref*{#1}}}

\let\pref=\prettyref
\let\Cref=\prettyref

\renewcommand{\eps}{\varepsilon}

\newcommand{\calD}{\mathcal D}

\renewcommand{\calN}{\mathcal N}

\renewcommand{\R}{\mathbb R}

\renewcommand{\Z}{\mathbb Z}

\usepackage{needspace}
\newlength{\ppartneed}
\newcommand{\ppart}[1]{%
  \par\Needspace{\ppartneed}\addvspace{\smallskipamount}%
  \noindent\textit{#1.}\hspace{0.5em}\ignorespaces%
}

\usepackage{algorithm}
\usepackage{algpseudocode}
\algnewcommand\algorithmicinput{\textbf{Input: }}
\algnewcommand\INPUT{\State\algorithmicinput}
\algnewcommand\algorithmicinitialize{\textbf{Initialize: }}
\algnewcommand\INIT{\State\algorithmicinitialize}
\algnewcommand\algorithmicrun{\textbf{Run: }}
\algnewcommand\RUN{\State\algorithmicrun}
\algnewcommand\algorithmicupdate{\textbf{Update: }}
\algnewcommand\UPDATE{\State\algorithmicupdate}
\algnewcommand\algorithmicset{\textbf{Set: }}
\algnewcommand\SET{\State\algorithmicset}
\algnewcommand\algorithmicquery{\textbf{Query: }}
\algnewcommand\QUERY{\State\algorithmicquery}
\algnewcommand\algorithmicoutput{\textbf{Output: }}
\algnewcommand\OUTPUT{\State\algorithmicoutput}

\makeatletter
\newcommand\appendix@section[1]{%
  \refstepcounter{section}%
  \orig@section*{\@Alph\c@section.\texorpdfstring{\,\,\,\;}{}#1}
}
\let\orig@section\section
\g@addto@macro\appendix{\let\section\appendix@section}
\makeatother

\renewcommand{\paragraph}[1]{\medskip\noindent{\bf #1{.}}}

\makeatletter
\newcommand{\saveequation}[2]{
  #2 \label{#1}
  \protected@write\@mainaux{}{\string\SAVEEQUATION{#1}{\unexpanded{\unexpanded{#2}}}}%
}
\newcommand{\savetagequation}[3]{
  #3 \label{#1} \tag{#2}
  \protected@write\@mainaux{}{\string\SAVEEQUATION{#1}{\unexpanded{\unexpanded{#3}}}}%
}
\newcommand{\SAVEEQUATION}[2]{%
  \global\@namedef{SAVEDEQUATION@#1}{#2}%
}
\newcommand{\repeatequation}[1]{%
  \ifcsname SAVEDEQUATION@#1\endcsname
    \@nameuse{SAVEDEQUATION@#1}\tag{\ref{#1}}%
  \else
    ?? \notag
  \fi
}
\makeatother

\setlist[itemize]{topsep=-4pt, partopsep=2pt}

\usepackage{thmtools}
\declaretheoremstyle[%
  spaceabove=-2pt,%
  spacebelow=6pt,%
  headfont=\normalfont\itshape,%
  postheadspace=1em,%
  qed=\qedsymbol%
]{mystyle} 
\declaretheorem[name={Proof},style=mystyle,unnumbered,
]{prf}

\togglefalse{focs}

\begin{document}

\author{Yutong Li}
\address{Department of Applied Mathematics and Statistics, Johns Hopkins University, USA}
\email{\href{mailto:yli685@jh.edu}{yli685@jh.edu}}

\author{Juspreet Singh Sandhu}
\address{Department of Computer Science, Colorado State University, USA}
\email{\href{mailto:jsinghsa@ucsc.edu}{js.sandhu@colostate.edu}}

\author{Jonathan Shi}
\address{Chipletics Inc, Redmond WA, USA}
\email{\href{mailto:jshi@cs.cornell.edu}{jshi@cs.cornell.edu}}

\title[On overlap concentration in the Curie--Weiss Random Field model]{On overlap concentration in the Curie--Weiss Random Field model}

\begin{abstract}
\small
\noindent We show that the overlaps of two independent replicas drawn from the Gibbs measure of the Curie--Weiss model with random field (CWRF) exhibit sub-Gaussian tails when the inverse temperature $ \beta < 1$ and the random field is drawn from a Gaussian distribution that makes the expected magnetization of the CWRF equal to its expected overlap in the thermodynamic limit. To show this, we prove finite-moment concentration via a rigorous Laplace approximation for a certain random rate function, and ``bootstrap'' it to a uniform bound for the moment-generating function of the squared overlap deviations using a sharp analysis of the associated rate function.
\end{abstract}

\maketitle

\thispagestyle{empty}
\vspace{-5mm}
\renewcommand{\baselinestretch}{0.9}\normalsize
{
  \hypersetup{linkcolor=Red}
  \setcounter{tocdepth}{1}
  \tableofcontents
}
\renewcommand{\baselinestretch}{1.0}\normalsize

%
%
%
%
%
%

\newpage 
\pagenumbering{arabic}

\section{Introduction}\label{sec:introduction}
The Curie--Weiss random field (CWRF) model is perhaps the most \emph{basic} disordered mean-field model over Ising spins. Its Hamiltonian is given as
\[
    H(\sigma) := \frac{\beta^2}{2n}\left(\sum_{i=1}^n\sigma_i\right)^2 + \sum_{i=1}^n h_i\sigma_i\, ,
\]
where $\beta > 0$ is the inverse temperature parameter, and $h_i \sim \calD$ are sampled i.i.d.\,from $\calD$. The macroscopic behavior (such as the free-energy's thermodynamic limit) is characterized by the magnetization of this model \cite[\S 14]{bovier2016metastability} but the addition of the external field $\mathbf{h}$ makes the local Gibbs measure somewhat asymmetric. Consequently, another natural order-parameter for the system is the overlap
\[
    R_{\sigma, \tau} = \frac{1}{n}\sum_{i=1}^n\sigma_i\tau_i\,,
\]
where $\sigma, \tau$ are two independent samples from the Gibbs measure conditioned on a draw of the external field $\mathbf{h} \sim \calD^{\ot n}$. The behavior of this statistic cannot generally be recovered from the magnetization alone. Prior work on the CWRF has given a fairly detailed account for the behavior of the magnetization (see \S\,\ref{subsec:related-work}). We take a complementary two-replica point of view: our aim is to identify and prove nonasymptotic concentration estimates for overlaps. Thus, our main result is an (annealed) bound demonstrating sub-Gaussians tails for the centered overlaps (\pref{thm:main}). 

\subsection{Motivation} This question is motivated by the role of overlaps throughout the theory of disordered spin systems, both mathematically \cite{parisi1980order,parisi1980sequence,talagrand2010mean,talagrand2010mean2} and (more recently) algorithmically \cite{huang2024sampling,davies2026potential}. 

\paragraph{Overlap concentration and high-temperature behavior} In mean-field spin glasses, concentration of the replica overlap is a central signature of the high-temperature regime (also known as the replica-symmetric regime). For instance, using a mixture of the Guerra interpolation and cavity interpolation, Talagrand's high-temperature analysis of the Sherrington--Kirkpatrick (SK) model bounds an exponential square moment of the centered overlaps \cite[Theorem~1.4.1]{talagrand2010mean}. A later result extends the same (qualitative) estimate into the interior of the replica-symmetric region under a global variational condition and strict AT stability \cite[Theorem~13.7.1]{talagrand2010mean2}. The conclusion proved in our result is of the same type: uniform exponential integrability of $n(R_{\sigma,\tau}-q^*)^2$ or, equivalently, sub-Gaussian tails for
$\sqrt n(R_{\sigma,\tau}-q^*)$.

\paragraph{Disorder chaos and superconcentration} Our main result about the tails of the overlap is also the natural endpoint of a question about overlaps drawn from two correlated copies of the model: 
\begin{center}
    \emph{If the centered Gaussian field is evolved by the Ornstein--Uhlenbeck semigroup, on sampling one replica before and another after the perturbation, how does the cross-overlap change with the noise-level?}
\end{center} 
For the SK model with no external field, Chatterjee used this coupling to prove disorder chaos and, through the Gaussian variance representation, a superconcentration bound for the free energy
\cite{chatterjee2009disorder,chatterjee2014superconcentration}. The analogy requires a careful interpretation in the case of the CWRF: at $\beta=0$ the CWRF becomes a random linear field, and its free energy is $\sum_i\log(2\cosh h_i)$, whose variance is of order $n$ whenever the field is nondegenerate. Thus, a disorder chaos question for the CWRF free energy would first have to separate this one-site disorder contribution. With this in mind, \pref{thm:main} supplies the required zero-perturbation baseline for a quantitative study of the disorder-perturbed overlaps. To the best of our knowledge, no previous work gives a characterization of the moment bounds for the overlap of the CWRF under two Ornstein--Uhlenbeck-correlated random fields.

\paragraph{Sampling algorithms} For the CWRF model, there is already strong high-temperature control of the dynamics. In the normalization used below, the interaction matrix is $\beta^2\mathbf 1\mathbf 1^{\mathsf T}/n$, so the spectral-gap theorem of Eldan, Koehler, and Zeitouni applies throughout $\beta<1$ and yields a
Poincar\'e inequality and fast Glauber mixing uniformly in the external field \cite{eldan2022spectral}. Note that for the CWRF, classical Dobrushin contraction applies independently in the same range, since the total single-site influence is strictly smaller than $\beta^2$ \cite{dobrushin1968description}. These
statements control relaxation and fluctuations around finite-volume quenched means. Their formulations do not, however, identify the deterministic center or state the annealed exponential-square estimate considered here. In contrast, our proof develops a three-dimensional random Laplace method that tracks replica magnetizations and their overlap simultaneously. More recently, overlap concentration has been shown to be critical in the analysis of sampling algorithms for mean-field spin-glasses in the high-temperature regime \cite{huang2024sampling,davies2026potential}.

\paragraph{Relationship to the planted SK model} The CWRF model is also closely related to the stochastic localization (SL) process \cite[(11)]{Eld13} for the SK model~\cite{davies2026potential}. Stochastic localization is a time-indexed process of increasingly localized distributions that can be characterized by the addition of a time-dependent linear \vocab{tilt} to the original model's Hamiltonian, noted for being the central object in a program to prove fast mixing for Markov chains~\cite{chen2022localization} and for its equivalence to diffusion models~\cite{Mon23}. Talagrand~\cite[Proposition 1.4.8 and Theorem 1.4.1]{talagrand2010mean} showed how to reduce overlap concentration in the SK model to overlap concentration in the random linear model, and applying an analogous line of reasoning to the planted representation of the SL-tilted SK model \cite{el2022sampling} yields the CWRF (with random field of the type considered in this paper) as the corresponding target under a Guerra interpolation \cite[\S 1.4]{talagrand2010mean}.

\subsection{Main result} In this paper, we fix $t \geq 0$. Let $\mathbf h^t \sim \calN(t\mathbf{1}_n,(\beta^2q^* + t)\Id_n)$ be the external field\footnote{\,This particular choice of field makes the resulting CWRF the target model for a Guerra interpolation applied to the planted representation of the SL-tilted SK model.}.
Let $(m^*,q^*)$ solve the fixed-point equations in \eqref{eq:fixpoint-equations}.
Our main result shows that the fluctuations of the overlaps for this CWRF model exhibit sub-Gaussian tails.

\begin{theorem}[CWRF has sub-Gaussian overlap concentration]\label{thm:main}
    For every $t \ge 0$ and $0< \beta < 1$, there exists a constant $\lambda > 0$ such that
    \[
        \E_{\mathbf h^t}\left[\an{e^{\lambda n \left(R_{\sigma, \tau}-q^*\right)^2}}_{\mathbf h^t}\right] \le 2\,,
    \]
    where $\an{\cdot}_{\mathbf{h}^t}$ denotes averaging with respect to two independent copies of the Gibbs measure for the CWRF given $\mathbf{h}^t$, and $\E_{\mathbf{h}^t}[\cdot]$ denotes an expectation with respect to $\calN(t\mathbf{1}_n,(\beta^2q^* + t)\Id_n)$.
\end{theorem}

\subsection{Notation}\label{subsec:Notation}
Let $(m^*,q^*)$ solve the fixed-point equations
\begin{equation}\label{eq:fixpoint-equations}
    m^* = \E_{h \sim \calN(0,1)}\left[\tanh\left(\beta^2 m^* + t + \sqrt{\beta^2 q^* + t}h\right)\right]\, , \text{ and }q^* = \E_{h \sim \calN(0,1)}\left[\tanh^2\left(\beta^2 m^* + t + \sqrt{\beta^2 q^* + t}h\right)\right]\, .
\end{equation}
The existence of $m^*$ and $q^*$ is guaranteed by \cite{davies2026potential}. 

Let $(\Omega,\mathcal A,\P)$ be a probability space supporting an i.i.d. sequence $(h_i^t)_{i\geq1}$ with distribution $h_i^t\sim\mathcal N\left(t,\beta^2q^*+t\right)$. For every $n\in\Z_{>0}$, define the random field $\mathbf h^t:=(h_1^t,\ldots,h_n^t)$. We also write $h^t\sim\mathcal N\left(t,\beta^2 q^* + t\right)$ for a generic random variable with the same distribution. Let $\mathbf h=(h_1,\ldots,h_n)\in\mathbb R^n$ be a fixed realization of the random field $\mathbf h^t$. In this paper, we frequently write ``for almost every realization $\mathbf h = \mathbf h^t(\omega)$'', which means there exists an event $\Omega_0 \in \mathcal{A}$ with $\mathbb{P}(\Omega_0) = 1$ such that $\omega \in \Omega_0$. Given a realization $\mathbf h$, the Gibbs measure is defined by
\[
\mu_{n,\mathbf h}(\sigma):=\frac{1}{Z_n(\mathbf h)}\exp\left\{\frac{\beta^2}{2n}\left(\sum_{i=1}^n\sigma_i\right)^2+\sum_{i=1}^nh_i\sigma_i\right\},\qquad \mathbf \sigma\in\{-1,+1\}^n,
\]
where
\[
Z_n(\mathbf h):=\sum_{\sigma\in\{-1,+1\}^n}\exp\left\{\frac{\beta^2}{2n}\left(\sum_{i=1}^n\sigma_i\right)^2+\sum_{i=1}^nh_i\sigma_i\right\}.
\]
Let $\sigma,\tau \stackrel{\mathrm{i.i.d.}}{\sim} \mu_{n,\mathbf h}$ and define the overlap between two spin configurations $\sigma$ and $\tau$ by $R_{\sigma,\tau}:=\frac{1}{n}\sum_{i=1}^n\sigma_i\tau_i$. For any integrable $f:\left(\{-1, +1\}^n\right)^2 \rightarrow \mathbb{R}$, define
\[
    \langle f\rangle_\mathbf h = \sum_{\sigma,\tau\in\{-1,+1\}^n}f(\sigma,\tau) \mu_{n,\mathbf h}(\sigma)\mu_{n,\mathbf h}(\tau) = \E_{(\sigma, \tau)\sim \mu_{n, \mathbf h}^{\otimes 2}}\left[f(\sigma, \tau)\right]\,.
\]
\subsection{Related work}\label{subsec:related-work} We briefly overview prior work on the CWRF model, as well as investigations about overlap concentration and disorder chaos in other canonical mean-field spin-glass models.

\paragraph{Magnetization: deviations and propagation of chaos}
The equilibrium magnetization of the CWRF has been studied at several scales. First, de Matos and Perez gave an early analysis of its fluctuations \cite{amaro1991fluctuations}. L\"owe, Meiners, and Torres then established a quenched large-deviation principle at speed $n$ with an explicit rate function \cite{lowe2013large}. Following up on this, L\"owe and Meiners also proved moderate-deviation principles around global minima, with the scale and rate reflecting the type of the minimum \cite{lowe2012moderate}. Most related to our techniques, Kabluchko and L\"owe proved quenched propagation of chaos \cite[Theorem~1.2]{kabluchko2024propagation}. Their quantitative argument also yields a uniform likelihood-ratio approximation when $k_n=o(n^{1/2-\eta})$ for every fixed $\eta>0$ \cite[Remarks~1.5(3) and~3.1]{kabluchko2024propagation}. Ben Arous and Zeitouni introduced increasing propagation of chaos for
mean-field Gibbs measures \cite{benarous1999increasing}. Their techniques prove that, for the Curie--Weiss model, the scale of relative-entropy convergence is only valid for $k_n=o(\sqrt n)$-coordinate marginals, and the critical fluctuation law obstructs product approximation at order $\sqrt n$.

\paragraph{Overlap concentration and disorder perturbations}
For the SK model, Talagrand proves exponential-square concentration of the overlap in a high-temperature regime
\cite[Theorem~1.4.1]{talagrand2010mean}. In \cite[\S\,13]{talagrand2010mean2}, Talagrand obtains the same qualitative bound inside the replica-symmetric region under both a global variational condition and strict a stability condition \cite[Theorem~13.7.1]{talagrand2010mean2}. Chatterjee's Ornstein--Uhlenbeck representation connects cross-overlaps under correlated disorder to free-energy variance and implies, in the zero-field SK model, disorder chaos and superconcentration \cite{chatterjee2009disorder,chatterjee2014superconcentration}. For random-field models, the picture is different: same-disorder overlap self-averaging has been proved for the finite-dimensional RFIM outside an exceptional parameter set \cite{chatterjee2015absence}.

\subsection{Proof overview}\label{subsec:proof-overview} The overall idea is to control the even-moments explicitly, obtain similar decay for odd moments via Jensen's inequality, and then ``bootstrap'' this to control the MGF by showing the control over the finite moments permits a geometric series convergence. The crux of the proof lies in two key ingredients: an implementation of a rigorous (random) Laplace method approximation scheme, and proofs about the local analytic properties of a certain three-dimensional rate function.

\paragraph{Uniqueness and equality of fixed-points} Before beginning the analysis to control the MGF, it is critical to establish that the fixed-points $(m^*,q^*)$ are unique (and equal) as they will be the critical points around which Taylor expansions in the Laplace approximation are conducted. By an analysis exactly akin to that of \cite[Proof of Lemma 5.32]{davies2026potential} with a slight tightening of a Lipschitz estimate using the properties of $\tanh(\cdot)$ in conjunction with the explicit form of the Gaussian density for $\calN(t,t)$, it is easy to conclude that $(m^*,q^*)$ exist, and that they are unique with $m^* = q^*$ when $0 < \beta < 1$ (\pref{lem:extend-nishimori-condition-beta-1}).

\paragraph{Controlling the even moments} We first rewrite $\left\langle\left(R_{\sigma,\tau}-q^*\right)^{2k}\right\rangle_{\mathbf h} $ for any fixed realization of the external field $\mathbf h=\mathbf h^t(\omega)$. Using the definitions of the normalized overlap and explanding the Gibbs average,
\[
\begin{aligned}
    \left\langle n^k\left(R_{\sigma,\tau}-q^*\right)^{2k}\right\rangle_{\mathbf h} &= \E_{(\sigma, \tau)\sim \mu_{n, \mathbf h}^{\otimes 2}}\left[\left(n\left(R_{\sigma,\tau}-q^*\right)^{2}\right)^k\right]\\
    &= k\int_0^\infty s^{k-1} \P_{(\sigma, \tau)\sim \mu_{n, \mathbf h}^{\otimes 2}}\left(n\left(R_{\sigma,\tau}-q^*\right)^2>s\right) \,ds\\
    &= k\int_0^\infty s^{k-1} \E_{(\sigma, \tau)\sim \mu_{n, \mathbf h}^{\otimes 2}}\left[ \mathbf 1_{\{n\left(R_{\sigma,\tau}-q^*\right)^2>s\}}\right] \,ds\\
    &= k\int_0^\infty s^{k-1} \left\langle \mathbf 1_{\{n\left(R_{\sigma,\tau}-q^*\right)^2>s\}} \right\rangle_{\mathbf h} \,ds.
\end{aligned}
\]
Obtaining this integral over a Gibbs averaged indicator function, we finally take the expectation over $\mathbf h^t(\omega)$. An application of Tonelli's theorem allows for the swapping of the integral and $\E_{\mathbf h^t}[\cdot]$, yielding
\begin{equation}\label{eq:expectation-and-integral}
    \E_{\mathbf h^t}\left[\left\langle\left(R_{\sigma,\tau}-q^*\right)^{2k}\right\rangle_{\mathbf h^t}\right]=\frac{k}{n^k}\int_0^\infty s^{k-1}\E_{\mathbf h^t}\left[\left\langle\mathbf 1_{\left\{n\left(R_{\sigma,\tau}-q^*\right)^2>s\right\}}\right\rangle_{\mathbf h^t}\right]\,ds.
\end{equation}
Note that for any $\lambda >0$ and fixed external-field $\mathbf h \in \R^n$,
\[
\left\langle\mathbf 1_{\{n\left(R_{\sigma,\tau}-q^*\right)^2>s\}}\right\rangle_\mathbf h\leq e^{-\lambda s}\left\langle e^{\lambda n\left(R_{\sigma,\tau}-q^*\right)^2} \right\rangle_\mathbf h,
\]
and therefore we need to establish an upper bound of $\left\langle e^{\lambda n\left(R_{\sigma,\tau}-q^*\right)^2} \right\rangle_\mathbf h$ and then average it over $\mathbf h \sim \calN(t,\beta^2q^* + t)$.

\paragraph{Rigorous Laplace method for random rate functions} A result of~\cite[Theorem 1.2]{kabluchko2024propagation} implies that not-too-large marginals of the CWRF are close in total variation distance to an explicit product measure, provided that $m^*$ (the magnetization) is a unique solution to \eqref{eq:fixpoint-equations}. As shown in~\pref{lem:extend-nishimori-condition-beta-1}, this uniqueness holds provided $\beta < 1$. It would be ideal to combine this ``nearly-product'' decomposition with an argument akin to the proof of~\cite[Theorem 1.4.8]{talagrand2010mean} for product measures to obtain overlap concentration and, therefore, control the term in \eqref{eq:expectation-and-integral}.

Unfortunately, the ``propagation of chaos''~\cite[Theorem 1.2]{kabluchko2024propagation} can only apply to marginals of size $O(n^{1/2-\eta})$ (for some small $\eta > 0$). Any naive attempt at using this result to split the $n$ coordinates into smaller blocks of that size, invoke the ``near-product'' distribution of each block to apply~\cite[Theorem 1.4.8]{talagrand2010mean} verbatim within the blocks, and then control variances across the blocks via conditional Gaussian MGF evaluations and Bernstein-type bounds does not seem to allow control of the exponential moment of the squared overlap deviation. 

Consequently, we are forced to do Laplace-method approximations around the critical point of the exponential moment of the squared overlap deviation after applying the Hubbard--Stratanovitch (HS) transform -- this extends the ideas of Kabluchko and L\"owe~\cite[\S3]{kabluchko2024propagation} to a three-dimensional analogue tracking \emph{two} fixed-point equations (\ref{eq:fixpoint-equations}). 
The Laplace approximation is done by first applying the Hubbard--Stratanovitch transform thrice to obtain
\[
\an{e^{\lambda n\left(R_{\sigma,\tau} - q^*\right)^2}}_\mathbf h = \sqrt{\frac{\lambda n}{\pi}} \frac{\int_{\R^3}e^{\bar F(\mathbf h; u,v,w)n}du\,dv\,dw}{\int_{\R^2}e^{\bar F(\mathbf h; u,v,0)n}du\,dv}\,,
\]
for certain explicit (rate) functions $\bar F(\mathbf h; u,v,w)$ (see \S\,\ref{sec:bounding-overlap-moments-for-beta-1/2}).
The critical region is when the magnetizations $(x,y)$ of the two replicas lie in $(\beta m^* - n^{-\delta}, \beta m^* + n^{-\delta})$ and the conditional (normalized) overlap is $\approx q^*$, for the explicit choice $(\beta m^*,\beta m^*,0)$ as a critical point. After this, the goal becomes to assert that the critical point is the global maximum. This is accomplished by estimating the integral via the Laplace method prediction, and bounding the error via uniform control of the third derivatives in conjunction with concentration arguments that control the suprema of sub-Gaussian fluctuations inside a reasonable sized annulus along with sharp decay for the tails. 

\paragraph{Analytic desiderata for the rate function}\label{subsec:multioverlaps-cwrf}  We use the following facts to permit a rigorous Laplace approximation for $0 <\beta < 1/2$ and $0 < \lambda < 7/\16$:
\allowdisplaybreaks
\begin{enumerate}[itemsep=0.5em]
    \item \textbf{(Unique global maximizer):} The function $\td F(u,v,w)$ (see \S\,\ref{sec:bounding-overlap-moments-for-beta-1/2}) has a unique global maximizer at $(\beta m^*,0,0)$.
    \item \textbf{(Negative definite-ness of $\nabla^2 \td F(\beta m^*,0,0)$):} $\nabla^2 \td F(u,v,w)_{u=\beta m^*, v = w = 0} \prec 0$.
    \item \textbf{(Uniformly bounded third-order derivatives):} $|\partial_{i,j,k} \bar F(\mathbf h; u,v,w)| \le C_{\beta,\lambda}\,, \forall\,i,j,k \in \{x,y,w\}^{3}$.
\end{enumerate}
Proving these estimates relies on performing various explicit calculus computations on the rate function, and probing their local analytic properties, and is done in \S\,\ref{sec:bounding-overlap-moments-for-beta-1/2}.
 To apply Gaussian-type estimates to the curvature term, it becomes critical to argue that a ``reduced'' Hessian is negative-definite, and some calculus then reveals that this limits $\beta < 1/2$ and $\lambda < 7/16$. The full random Laplace estimation strategy is carried out in the proof of \pref{thm:whp-overlap-concentration-laplace-cwrf}. Showing that the bound in \pref{thm:whp-overlap-concentration-laplace-cwrf} holds with sufficiently high probability (\pref{lem:probability-good-event}) allows for the use of explicit calculus along with the fact that the even moments give rise to a geometric series to bound the MGF of the squared overlap deviations (\pref{thm:uniform-exponential-moment-bound-for-0-beta-1/2}).

\paragraph{Extending the analysis to $\beta < 1$ via refined convexity} The preliminary proof analyzing the rate function asks for global concavity in $v$, including at extreme values of the auxiliary overlap coordinate $w$ that does \emph{not} govern the Laplace integral. For very negative $w$, the two-spin tilt strongly favors $\sigma\tau=-1$ and the $v$ direction has significantly less influence. Consequently, once $\beta>1/\sqrt2$, it can even acquire positive curvature, and so a global curvature certificate cannot cover all $\beta<1$ (see \S\,\ref{subsec:counterexample} for a counter-example).

In \pref{sec:extending-to-beta-1}, we pivot the strategy to use the stationarity equation in $w$. If $(x^*,y^*,w^*)$ is a global maximizer of the rate function, then
\begin{equation}\label{eq:slab-overview}
    w^*+q^*=\E_{h^t}\E_{\pi_{h^t}}[\sigma\tau]\in[-1,1],
    \qquad w^*\in[-1-q^*,1-q^*]\subseteq[-2,1],
\end{equation}
where
\[
    \pi_h(\sigma,\tau)\propto
    \exp\{(\beta x+h)\sigma+(\beta y+h)\tau
                   +2\lambda w\sigma\tau\}\, ,
\]
denotes a Gibbs-type density induced by the three-dimensional rate function. Therefore, it is enough to control the Hessian on the convex slab $\R^2\times[-1-q^*,1-q^*]$ which avoids the large values of $w$ where counterexamples arise (\pref{thm:uniqueness-global-maximizer-beta-less-than-one}). We now give a rough sketch of the argument that analyzes the Hessian: in the original $(x,y,w)$ coordinates, some calculus reveals that the Hessian can be written as a block matrix
\[
    \nabla^2\hat F=
    \begin{pmatrix}A&\mathbf b\\ \mathbf b^{\mathsf T}&d\end{pmatrix}\,,
\]
and
\begin{align*}
    A&=-I_2+\beta^2\E_{h^t}\left[
       \Cov_{\pi_{h^t}}\binom{\sigma}{\tau}\right],\\
    \mathbf b&=2\beta\lambda\E_{h^t}\left[
       \binom{\Cov_{\pi_{h^t}}(\sigma,\sigma\tau)}
             {\Cov_{\pi_{h^t}}(\tau,\sigma\tau)}\right],\\
    d&=-2\lambda+4\lambda^2\E_{h^t}
       \left[\Var_{\pi_{h^t}}(\sigma\tau)\right].
\end{align*}
The key point is that, for the two-spin law above,
\[
    |\Cov_{\pi_h}(\sigma,\tau)|
    \leq\tanh(2\lambda|w|)
    \leq\tanh(4\lambda)\leq4\lambda\,,
\]
throughout the convex slab. Together with Cauchy--Schwarz, this gives
\[
    -A\succeq(1-\beta^2-4\beta^2\lambda)I_2,
    \qquad \norm{\mathbf b}_2^2\leq8\beta^2\lambda^2,
    \qquad\text{and\, } -d\geq2\lambda-4\lambda^2.
\]
Choosing
\[
    0<\lambda<\min\left\{\frac14,
                \frac{1-\beta^2}{16\beta^2}\right\},
\]
the Schur complement absorbs the mixed block and makes
$-\nabla^2\hat F$ positive definite throughout the slab. Since every global maximizer lies in the slab and $\td c$ is stationary there, $\td c$ is the unique global maximizer. Using a proof similar to that in \pref{sec:bounding-overlap-moments-for-beta-1/2}, we extend the result of \pref{thm:uniform-exponential-moment-bound-for-0-beta-1/2} to the case $\beta < 1$ to obtain \pref{thm:uniform-exponential-moment-bound-for-0-beta-1}, which proves \pref{thm:main}.

\section{Uniqueness of fixed-point solutions for \texorpdfstring{$0 <\beta < 1$}{0 < beta < 1}}\label{sec:uniqueness-of-fixed-point-solutions} In this section, we show $(m^*, q^*)$ are unique to the fixed-point equations for $0<\beta <1$. Furthermore, as an additional conclusion of the proof, we have $m^* = q^*$, which is an extension of \cite[Lemma 5.32]{davies2026potential}.

\begin{lemma}[Uniqueness of fixed-points for $0<\beta<1$]\label{lem:extend-nishimori-condition-beta-1}
Assume $0<\beta<1$, and let $(m^*,q^*)$ solve the fixed-point equations as defined in~\pref{eq:fixpoint-equations}, then $m^*$ and $q^*$ are unique solutions to the fixed-point equations, in addition, $m^*=q^*$.
\end{lemma}
\begin{prf}
Let $h \sim \mathcal N(0,1)$, $\alpha=\beta^2 q^* + t\geq 0$, and $\delta=\beta^2\left(m^*-q^*\right)$, we have $\beta^2 m^* + t=\alpha+\delta$ and $\sqrt{\beta^2 q^* + t}=\sqrt{\alpha}$. Define $g(x):=\tanh (x)-\tanh ^2(x)$. Then, the fixed-point equations can be rewritten as
\[
m^*=\E_{h}[\tanh (\alpha+\delta+\sqrt{\alpha} h)], \quad q^*=\E_{h} \left[\tanh ^2(\alpha+\delta+\sqrt{\alpha} h)\right].
\]

By \cite[Lemma 5.32]{davies2026potential}, we have $\E_{h}\left[g(\alpha+\sqrt{\alpha}\,h)\right]=0$. As a consequence,
\[
m^*-q^*=\E_{h}\left[g(\alpha+\delta+\sqrt{\alpha}\,h)\right]-\E_{h}\left[g(\alpha+\sqrt{\alpha}\,h)\right].
\]

First, we show $m^*$ is unique and $m^*\geq 0$. Define $\varphi(m):=\E_{h}\left[\tanh\left(t+\beta^2m+\sqrt{\alpha}\,h\right)\right]$, the fixed-point equation gives $m^*=\varphi(m^*)$. Moreover,
\[
\varphi^\prime(m)=\beta^2\E_{h}\left[\operatorname{sech}^2\left(t+\beta^2 m+\sqrt{\alpha}\,h\right)\right]\leq \beta^2<1.
\]
Hence $\varphi(m)-m$ is strictly decreasing, which implies $m^*$ is unique. Pairing $h=z$ with $-z$, we have
\[
\tanh(t+\sqrt{\alpha} z)+\tanh(t-\sqrt{\alpha} z) = \frac{\sinh(2t)}{\cosh(t+\sqrt{\alpha} z)\cosh(t \sqrt{\alpha} z)} \geq 0,
\]
then
\[
\varphi(0) - 0 = \E_h\left[\tanh\left(t+\sqrt{\alpha}\,h\right) \right] \geq 0.
\]
Therefore $m^*\ge 0$ and $\alpha+\delta=t+\beta^2 m^*\geq 0$.

Next, we show $\left|\partial_r \E_{h}\left[g(r+\sqrt{\alpha}\,h)\right]\right|\leq 1$. Taking the derivative of $g(x)$, we have
\[
g'(x)=\operatorname{sech}^2 x\,(1-2\tanh x)=(1-\tanh^2 x)(1-2\tanh x) \in (-1, 3).
\]
Then
\[
\partial_r \E_{h}\left[g(r+\sqrt{\alpha}\,h)\right] = \E_{h}\left[g^\prime(r+\sqrt{\alpha}\,h)\right] > -1,
\]
which shows the lower bound. To prove the upper bound, first consider the case $\alpha=0$. Since $r\geq 0$, we immediately have $g'(r)\leq 1$. Now suppose that $\alpha>0$, and let $X\sim\mathcal{N}(r,\alpha)$ with density $p(x)$. For $x\geq 0$, define
\[
\rho(x):=\frac{p(-x)}{p(x)}
=\exp\left(-\frac{2rx}{\alpha}\right)\in(0,1].
\]
Pairing the contributions from $x$ and $-x$, we obtain
\[
\begin{aligned}
\E_{X}\left[g'(X)-1\right]&=\int_0^\infty p(x)\left[\bigl(g'(x)-1\bigr)+\rho(x)\bigl(g'(-x)-1\bigr)\right]\,dx\\
&=\int_0^\infty p(x)\left[(1-\rho(x))\bigl(g'(x)-1\bigr)+\rho(x)\bigl(g'(x)+g'(-x)-2\bigr)\right]\,dx\\
&\leq 0,
\end{aligned}
\]
where the last equality holds since
\[
g'(x)+g'(-x)-2= 2\operatorname{sech}^2(x)-2 = -2\tanh^2(x) \leq 0.
\]

Finally, by the fundamental theorem of calculus,
\[
m^*-q^*=\int_\alpha^{\alpha+\delta} \partial_r \E_{h}\left[g(r+\sqrt{\alpha}\,h)\right] \,dr=\int_\alpha^{\alpha+\delta}\E_{h}\left[g'(r+\sqrt{\alpha}\,h)\right]\,dr.
\]
Since $\alpha \geq 0$ and $\alpha+\delta \geq 0$, the whole integration path lies in $[0,\infty)$. Therefore,
\[
\left|m^*-q^*\right|\leq \int_{\min\{\alpha,\alpha+\delta\}}^{\max\{\alpha,\alpha+\delta\}}\left|\E_{h}\left[g'(r+\sqrt{\alpha}\,h)\right]\right|\,dr \leq |\delta|=\beta^2 \left|m^*-q^*\right|.
\]
This forces $|m^*-q^*|=0$. Then, $q^*$ is also unique solution to the fixed-point equations.
\end{prf}

\section{Annealed bound for the MGF of the squared overlap deviations at \texorpdfstring{$0 < \beta < 1/2$}{0 < beta < 1/2}}\label{sec:bounding-overlap-moments-for-beta-1/2}
In this section, we prove upper bounds for the even moments and the exponential moment of the squared overlap deviation when $0<\beta < 1/2$. The analysis here will begin by writing the quenched MGF for the squared overlap deviation in terms of explicit (random) rate functions, and then proving a list of analytic desiderata about their deterministic counterparts. Once this is obtained, the key analytic estimate of the section is conducted in \pref{thm:whp-overlap-concentration-laplace-cwrf} which is a rigorous (random) Laplace method. This is eventually boosted to an annealed estimate on the MGF of the squared overlap deviations in \pref{thm:uniform-exponential-moment-bound-for-0-beta-1/2}.

\subsection{Explicit computation of the rate function via the Hubbard--Stratanovitch transform}\label{sec:hs-rate} We are interested in applying a Hubbard-Stratonovich transform to the exponential moment of the squared overlap deviation, so as to transform it into a three-dimensional integral over some rate-like function $F(x,y,w)$ over a pair of magnetizations $(x,y)$ and overlap-deviation $w$. Then we do a change-of-basis $(x, y) \to (u, v)$ that transforms $F(x,y,w) \to \bar F(u,v,w)$ so that $\bar F$ can be approximated via a Laplace method around its critical point $\td c := (\beta m^*,0 , 0)$.
        
\begin{lemma}[Nested Hubbard-Stratonovich transform for exponential moment of the squared overlap deviation]\label{lem:hs-variance-mgf}
    For any $\lambda, \beta >0$, and realization $\mathbf h = \mathbf h^t(\omega)$,
    \[
        \an{e^{\lambda n\left(R_{\sigma,\tau} - q^*\right)^2}}_\mathbf h = \sqrt{\frac{\lambda n}{\pi}}\frac{\int_{(x,y,w)\in\R^3}e^{F(\mathbf h; x,y,w)n}dx\,dy\,dw}{\int_{(x,y)\in\R^2}e^{F(\mathbf h; x,y,0)n}dx\,dy}\, ,
    \]
    where $F: \R^3 \to \R$ is given as
    \begin{equation}\label{eq:empirical-averages-function}
    \begin{aligned}
        F(\mathbf h; x,y,w) &:= -\frac{1}{2}\left(x^2 + y^2 + 2\lambda w^2\right) - 2\lambda w q^* + \\
        &\qquad \qquad \frac{1}{n}\sum_{i=1}^n \ln\left(e^{2\lambda w}\cosh\left(\beta x + \beta y + 2h_i \right) + e^{-2\lambda w}\cosh\left(\beta x - \beta y\right)\right)\,.
    \end{aligned}
    \end{equation}
\end{lemma}
\begin{prf}
    The proof follows by three applications of the Hubbard-Stratonovich transform. The technical steps to obtain the rate function $F(\mathbf h; x,y,w)$ are similar to those of~\cite[\S3]{kabluchko2024propagation}.
    \allowdisplaybreaks
    \begin{align*}
        &\an{e^{\lambda n\left(R_{\sigma,\tau} - q^*\right)^2}}_\mathbf h = \frac{1}{Z_{n, \mathbf h}^2}\sum_{\sigma,\tau}e^{\frac{\beta^2}{2n}(\sum_i \sigma_i)^2 + \sum_{i=1}^n h_i\sigma_i}e^{\frac{\beta^2}{2n}(\sum_i\tau_i)^2 + \sum_{i=1}^n h_i\tau_i}e^{\lambda n \left(R_{\sigma,\tau} - q^*\right)^2} \\
        &= \frac{1}{Z_{n, \mathbf h}^2}\frac{n}{2\pi}\sum_{\sigma,\tau}\left( \int_x e^{-x^2n/2 + \beta x\sum_i\sigma_i + \sum_i h_i\sigma_i}\,dx\right)\left(\int_y e^{-y^2n/2 + \beta y\sum_i\tau_i + \sum_i h_i\tau_i}\,dy\right)e^{\lambda n\left(R_{\sigma,\tau} - q^*\right)^2} \\
        &= \frac{1}{Z_{n, \mathbf h}^2}\left\{\sqrt{\frac{\lambda n^{3}}{4\pi^3}}\sum_{\sigma,\tau}\left(\int_{x,y,w}e^{\left(-x^2/2 -y^2/2-\lambda w^2\right)n}e^{-2\lambda q^*n w}e^{\sum_{i=1}^n\left(h_i(\sigma_i + \tau_i) + \beta (x\sigma_i + y\tau_i) + 2\lambda \sigma_i\tau_i w\right)}\right)dx\,dy\,dw\right\} \\
        &= \int_{\R^3}\frac{1}{Z_{n, \mathbf h}^2}\left\{\sqrt{\frac{\lambda n^3}{4\pi^3}}e^{\left(-(x^2 + y^2 + 2\lambda w^2)/2 - 2\lambda q^* w\right)n}\prod_{i=1}^n\left(\sum_{\sigma_i,\tau_i\in\{-1,1\}^2}e^{ h_i(\sigma_i + \tau_i) + \beta (x\sigma_i + y\tau_i) + 2\lambda \sigma_i\tau_i w}\right)dx\,dy\,dw\right\} \\
        &= \int_{\R^3}\frac{1}{Z_{n, \mathbf h}^2}\Bigg\{\sqrt{\frac{\lambda n^3}{4\pi^3}}e^{\left(-(x^2 + y^2 + 2\lambda w^2)/2 - 2\lambda q^* w\right)n} \\
        &\qquad\qquad\qquad\quad \prod_{i=1}^n\left(e^{2\lambda w}\left(e^{2 h_i + \beta(x+y)}+e^{-2 h^t_i -\beta(x+y)}\right) + e^{-2\lambda w}\left(e^{\beta(x-y)} 
        + e^{-\beta(x-y)}\right)\right)dx\,dy\,dw\Bigg\} \\
        &= \int_{\R^3}\frac{1}{Z_{n, \mathbf h}^2}\Bigg\{\sqrt{\frac{\lambda n^3}{4\pi^3}}e^{\left(-(x^2 + y^2 + 2\lambda w^2)/2-2\lambda q^* w\right)n} \\
        &\qquad\qquad\qquad\quad\prod_{i=1}^n2\left(e^{2\lambda w}\cosh\left(\beta x + \beta y + 2 h_i\right) + e^{-2\lambda w}\cosh\left(\beta x - \beta y\right)\right)dx\,dy\,dw\Bigg\} \\
        &= \int_{\R^3}\frac{2^n}{Z_{n, \mathbf h}^2}\left\{\sqrt{\frac{\lambda n^3}{4\pi^3}}e^{\left(-(x^2 + y^2 + 2\lambda w^2)/2-2\lambda q^* w\right)n}e^{\left(\frac{1}{n}\sum_i\ln\left(e^{2\lambda w}\cosh\left(\beta x + \beta y +2 h_i \right) + e^{-2\lambda w}\cosh\left(\beta x - \beta y\right)\right)\right)n}dx\,dy\,dw\right\} \\
        &= \frac{2^n}{Z_{n, \mathbf h}^2}\sqrt{\frac{\lambda n^3}{4\pi^3}}\int_{\R^3}e^{F(\mathbf h; x,y,w)n}dx\,dy\,dw\,.
    \end{align*}
    A similar Hubbard-Stratonovich transformation on the two magnetization variables yields
    \[
        Z_{n, \mathbf h}^2 = 2^n\frac{n}{2\pi} \int_{\R^2}e^{\left(-(x^2 + y^2)/2 + \frac{1}{n}\sum_i \ln\left(\cosh(\beta x + \beta y +2 h_i)+ \cosh(\beta x - \beta y)\right)\right)n}dx\,dy = 2^n\frac{n}{2\pi}\int_{\R^2}e^{F(\mathbf h; x,y,0)n}dx\,dy\, ,
    \]
    since $\cosh(x+y) + \cosh(x-y) = 2\cosh(x)\cosh(y)$. This immediately implies that
    \[
        \an{e^{\lambda n (R_{\sigma,\tau}-q^*)^2}}_\mathbf h = \sqrt{\frac{\lambda n}{\pi}}\frac{\int_{\R^3}e^{F(\mathbf h; x,y,w)n}dx\,dy\,dw}{\int_{\R^2}e^{F(\mathbf h; x,y,0)n}dx\,dy}\,. \qedhere
    \]
\end{prf}

\subsection{Analysis of the deterministic rate function}\label{subsec:hubbard-Stratonovich-transform-to-MGF-of-overlaps}

\begin{definition}[Deterministic rate function]\label{def:deterministic-rate-function}
    Define $\hat{F} : \R^3 \to \R$ for any $\lambda, \beta >0$ as
    \begin{align*}
        \hat{F}(x,y,w) &= -\frac{1}{2}\left(x^2 + y^2 + 2\lambda w^2\right) - 2\lambda q^* w \\
        &\qquad +\E_{h^t}\left[\ln\left(e^{2\lambda w}\cosh(\beta x+\beta y + 2h^t) + e^{-2\lambda w}\cosh(\beta x -\beta y)\right)\right].
    \end{align*}
\end{definition}

This naturally induces the ``deviation'' function, which measures the difference between empirical averages and the true expectation.
\begin{definition}[Deviation function]
\label{def:laplace-deviation-function}
    Let the function $\Delta: \R^3 \to \R$ measure the deviation between the ``ideal'' rate function $\hat{F}(x,y,w)$ and empirical rate function $F(\mathbf h; x,y,w)$, and be defined as
    \begin{align*}
        \Delta(\mathbf h; x,y,w) &= \frac{1}{n}\sum_{i=1}^n \ln\left(e^{2\lambda w}\cosh\left(\beta x + \beta y + 2h_i\right) + e^{-2\lambda w}\cosh\left(\beta x -\beta y\right)\right) \\
        &\qquad - \E_{h^t}\left[\ln\left(e^{2\lambda w}\cosh(\beta x + \beta y + 2h^t) + e^{-2\lambda w}\cosh(\beta x -\beta y)\right)\right]\,.
    \end{align*}
\end{definition}

For convenience, we do a change-of-basis $(x,y) \to (u, v) := ((x+y)/2,(x-y)/2)$ that transforms the function $F(\mathbf h; x,y,w) \to \bar{F}(\mathbf h; u,v,w)$ and $\hat{F}(x,y,w) \to \td{F}(u,v,w)$ so that the rate function decouples the $\cosh(\cdot)$ terms to be functions of two separate, independent variables. Immediately, we have the following definitions:

\begin{definition}[Empirical averages and Deterministic rate function in $(u, v,w)$ basis]
\label{def:empirical-averages-deterministic-rate-function-u,v,w}
    For any $\lambda, \beta >0$, define $\bar F: \R^3 \to \R$ as 
    \[
    \bar F(\mathbf h; u,v,w) = -\left(u^2 + v^2 + \lambda w^2\right) - 2\lambda q^* w + \frac{1}{n} \sum_{i=1}^n \ln\left(e^{2\lambda w}\cosh\left(2\beta u  + 2 h_i\right) + e^{-2\lambda w}\cosh\left(2\beta v\right)\right)\,,
    \]
    and $\td F: \R^3 \to \R$ as
    \[
    \td F(u,v,w) = -\left(u^2 + v^2 + \lambda w^2\right) - 2\lambda q^* w + \E_{h^t} \left[ \ln\left(e^{2\lambda w}\cosh\left(2\beta u  + 2 h^t\right) + e^{-2\lambda w}\cosh\left(2\beta v\right)\right)\right]\,.
    \]
\end{definition}

Similarly, we define the deviation function in $(u, v,w)$ basis:

\begin{definition}[Deviation function in $(u,v,w)$ basis]
\label{def:laplace-deviation-function-u,v,w}
    Let the function $\td \Delta: \R^3 \to \R$ measure the deviation between the ``ideal'' rate function $\td F(u,v,w)$ and empirical rate function $\bar F(\mathbf h; u,v,w)$, and be defined as
    \begin{align*}
        \td \Delta(\mathbf h; u,v,w) &= \frac{1}{n}\sum_{i=1}^n \ln\left(e^{2\lambda w}\cosh\left(2\beta u + 2h_i\right) + e^{-2\lambda w}\cosh\left(2 \beta v\right)\right) \\
        &\qquad \qquad  - \E_{h^t}\left[\ln\left(e^{2\lambda w}\cosh(2 \beta u + 2h^t) + e^{-2\lambda w}\cosh(2 \beta v)\right)\right]\,.
    \end{align*}
\end{definition}

\paragraph{A critical point of $\td{F}(u,v,w)$} We first prove that $\td c$ is a critical point for the function $\td{F}(u,v,w)$.

\begin{proposition}[$\td c$ is a critical point for $\td{F}(u,v,w)$]\label{prop:stationary-point-rate-function-variance-mgf}
Let $\td F: \R^3 \to \R$ be as defined in~\pref{def:empirical-averages-deterministic-rate-function-u,v,w} and $(m^*,q^*)$ be as defined in~\pref{eq:fixpoint-equations}. Then, for every $\lambda, \beta >0$,
\[
    \nabla \td F(\td c)=(0, 0, 0).
\]
\end{proposition}
\begin{prf}
    Define
    \begin{equation}\label{eq:function-of-u,v,w,h^t}
    A(h^t; u, v, w) = e^{2\lambda w} \cosh(2\beta u +2h^t)+e^{-2\lambda w} \cosh(2\beta v).
    \end{equation}
    Since $|\sinh(\cdot)|\leq \cosh(\cdot)$,
    \[
    \begin{aligned}
        &\left|\partial_u \ln A(h^t; u, v, w)\right| = \frac{2 \beta e^{2 \lambda w} |\sinh (2 \beta u+2 h^t)|}{e^{2\lambda w} \cosh(2\beta u +2h^t)+e^{-2\lambda w} \cosh(2\beta v)} \leq 2\beta , \\
        & \left|\partial_v \ln A(h^t; u, v, w)\right| = \frac{2 \beta e^{-2 \lambda w} |\sinh (2 \beta v)|}{e^{2\lambda w} \cosh(2\beta u +2h^t)+e^{-2\lambda w} \cosh(2\beta v)} \leq 2\beta , \\
        & \left|\partial_w \ln A(h^t; u, v, w)\right| =  \frac{2 \lambda \left|e^{2 \lambda w} \cosh (2 \beta u+2 h^t)-e^{-2 \lambda w} \cosh (2 \beta v)\right|}{e^{2\lambda w} \cosh(2\beta u +2h^t)+e^{-2\lambda w} \cosh(2\beta v)} \leq 2 \lambda .
    \end{aligned}
    \]
    By dominated convergence theorem, 
    \begin{align*}
        \partial_u \td{F}(\td c)&= 2\beta\left(-m^* + \E_{h^t} \left[ \frac{\sinh(2\beta^2 m^* +2h^t)}{\cosh(2\beta^2 m^* +2h^t)+1}\right]\right) =2\beta\left(-m^* + \E_{h^t} \left[\tanh\left(\beta^2 m^*+h^t\right)\right]\right)\,, \\
        \partial_v \td{F}(\td c) &= 2\beta\left(-0 + \E_{h^t} \left[\frac{\sinh(0)}{\cosh(2\beta^2 m^* +2h^t)+1}\right]\right) = 0\,, \\
        \partial_w \td{F}(\td c) &= -2\lambda q^* + 2\lambda\E_{h^t} \left[\frac{\cosh(2\beta^2 m^* +2h^t)-1}{\cosh(2\beta^2 m^* +2h^t)+1}\right] = 2\lambda\left(-q^* + \E_{h^t} \left[\tanh^2\left(\beta^2 m^* + h^t\right)\right]\right)\,.
    \end{align*}
    Since $m^*$ and $q^*$ are chosen to satisfy the fixed-point equations defined in~\pref{eq:fixpoint-equations}, the claim follows.
\end{prf}

A similar proof shows that $\hat c := (\beta m^*, \beta m^*, 0)$ is a critical point of $\hat F$.

\paragraph{Uniqueness of $\td c$ as global maximizer for $(\beta,\lambda) \in (0,1/2)\times(0,7/16)$} We begin by establishing the fact $\td c$ is a global maximizer for $\td F$. To do this we establish coercivity, and show that the function diverges to $-\infty$ as $(u, v, w) \to \infty$, which implies the existence of global maximizer. After this, we show the global maximizer is restricted to lie in a $2$-dimensional subspace. We conclude by establishing strict convexity in that subspace using perturbative arguments. We begin with analyzing the limit behavior of $\bar F$ and $\td F$.

\begin{proposition}[Behavior of $\bar{F}(\mathbf h; u,v,w)$ and $\td F(u, v,w)$ near infinite limits $\implies$ global maximizers exist]\label{prop:f-infinite-limits}
    $\bar F$ and $\td{F}$ are defined in~\pref{def:empirical-averages-deterministic-rate-function-u,v,w}. Then, for any $\beta, \lambda>0$,
    \[
        \lim_{\norm{(u,v,w)}_2 \to \infty} \bar F(\mathbf h; u,v,w) \overset{a.s.}{=} -\infty\,,\quad\text{and}\,,\quad  \lim_{\norm{(u,v,w)}_2 \to \infty} \td{F}(u,v,w) = -\infty\,,
    \]
    and so for almost every $\mathbf h(\omega)$, $\bar F$ and $\td F$ both have (at least one) global maximizer in $\R^3$.
\end{proposition}
\begin{prf}
    By using the fact that $\ln(e^{x}a + e^{-x}b) \le |x| + \ln(a + b)$ for any $x \in \R$ and $a, b \ge 0$, observe that
    \begin{align*}
        \ln\left(e^{2\lambda w}\cosh(2\beta u+ 2h_i) + e^{-2\lambda w}\cosh(2\beta v\right)
        &\leq 2\lambda |w| + \ln\left(\cosh\left(2\beta u+ 2h_i\right) + \cosh\left(2\beta v\right)\right) \\
        &= 2\lambda |w| + \ln2 + \ln\left(\cosh\left(\beta u+ \beta v + h_i)\right)\right) \\
        &\qquad \qquad + \ln\left(\cosh\left(\beta u - \beta v + h_i\right)\right)\,.
    \end{align*}
    This immediately implies that,
    \[
        \bar F(\mathbf h; u,v,w) \leq - \left(u^2 + v^2 +\lambda w^2\right) -2\lambda q^* w +2\lambda |w| + 2\beta |u| +2\beta |v| + \ln2 + \frac{2}{n}\sum_{i=1}^n |h_i| ,   
    \]
    where $\ln(\cosh(\cdot)) \le |\cdot|$ and triangle inequality are used. Let $\Phi$ be cumulative distribution function of $\calN(0,1)$. Since $h_i^t \stackrel{\mathrm{i.i.d.}}{\sim} \mathcal N\left(t,\beta^2 q^* + t\right)$, as $\beta^2 q^* + t\neq 0$,
    \[
    \E_{\mathbf h^t}|h_1^t| = \sqrt{\beta^2 q^* + t} \sqrt{\frac{2}{\pi}} \exp{\left(-\frac{t^2}{2(\beta^2 q^* + t)}\right)} + |t|\left(2\Phi\left(\frac{|t|}{\sqrt{\beta^2 q^* + t}}\right)-1\right) <\infty.
    \]
    By Kolmogorov's strong law-of-large numbers, 
    \[
    \frac{1}{n} \sum_{i=1}^n |h^t_i| \longrightarrow \E_{\mathbf h^t} |h_1^t| < \infty \quad \text{ almost surely.}
    \]
    Equivalently, for almost every realization $\mathbf h(\omega)$, we have $\sum_{i=1}^n|h_i|/n\to \mathbb E_{\mathbf h^t}|h_1^t|$. Then there exists $n_1(\omega)$, such that as $n>n_1(\omega)$, $\sum_{i=1}^n |h_i|/n \leq \E_{\mathbf h^t}|h_1^t| + 1$. Hence for any $n>0$ and almost every $\omega$, 
    \[
    \frac{1}{n} \sum_{i=1}^n |h_i| \leq \max{\left(|h_1|, \frac{1}{2}(|h_1|+|h_2|), \dots, \frac{1}{n_1(\omega)} \sum_{i=1}^{n_1(\omega)} |h_i|, \E_{\mathbf h^t} |h_1^t| + 1\right)} = O(1),
    \]
    this implies
    \[
        \bar F(\mathbf h; u,v,w) \le -\left(u^2 + v^2 + \lambda w^2\right) + 2\lambda(1+q^*)|w|+2\beta(|u|+|v|)+ O(1)\,.
    \]
    In the degenerate case, both $t$ and $q^*$ are zero. In this instance, we still have $\sum_{i=1}^n |h_i|/n = 0 = O(1)$.
    
    By Young's inequality, $\beta|u| \leq u^2/4 + \beta^2$, $\beta|v| \leq v^2/4 +\beta^2$, and $2\lambda(1+q^*)|w|\leq \lambda w^2/2 + 2\lambda(1+q^*)^2$, then for almost every realization $\mathbf h(\omega)$,
    \[
    \bar F(\mathbf h; u, v, w) \leq -\frac{1}{2}u^2 - \frac{1}{2}v^2 -\frac{\lambda}{2}w^2 + O(1) \to -\infty .
    \]
    Since $\bar F(\mathbf h; u, v, w) \to -\infty$, $\exists R_0>0$, such that when $\norm{(u, v, w)}_2 \geq R_0$, $\bar F(\mathbf h; u, v, w) < \bar F(\mathbf h; 0, 0, 0) - 1$. Therefore, the global maximizer for $\bar F$ lies in the closed ball $\bar{B_{R_0}(0,0,0)}$ centered at $(0,0,0)$ with radius $R_0$. Since $\bar F$ is continuous and $\bar{B_{R_0}(0,0,0)}$ is compact set, the existence of a global maximizer follows straightforwardly by the extreme value theorem. 
    
    The same argument, without appeal to the strong law-of-large numbers, applies to $\td{F}(u,v,w)$ as
    \begin{align*}
    \left|\E_{h^t}\left[\ln\left(e^{2\lambda w}\cosh(2\beta u + 2h^t) + e^{-2\lambda w}\cosh(2\beta v)\right)\right]\right| & \le 2\beta(|u|+|v|) +  2\lambda|w| + 2\E_{h^t}|h^t| + \ln 2 \\
    &\leq 2\beta\left(|u| + |v|\right) + 2\lambda|w| + O(1)\,.
    \end{align*}
    The proof is analogous to that for $\bar F$; continuity of $\td{F}$ follows by dominated convergence theorem.
\end{prf}

The goal now is to show that, for any choice of $(u,w)$ the map $v \to \td{F}(u,v,w)$ is \emph{strictly} concave in $v$. Since $\td{F}(u,v,w)$ is even in $v$, this forces $v=0$ at any maximizer.

\begin{lemma}[Strict concavity of $\td{F}(u,v,w)$ in $v$ for any $(u,w)$ with $0 < \beta < 1/2$]\label{lem:strict-concavity-f-transformed-basis}
   The function $\td{F}(u,v,w)$ defined in~\pref{def:empirical-averages-deterministic-rate-function-u,v,w} is \emph{even} and \emph{strictly concave} in $v$ for $0 < \beta < 1/2$ and any $\lambda >0$.
\end{lemma}
\begin{prf}
    Recall in~\pref{def:empirical-averages-deterministic-rate-function-u,v,w} we have
    \[
        \td{F}(u,v,w) = -\left(u^2 + v^2 + \lambda w^2\right) - 2\lambda q^* w + \E_{h^t}\left[\ln\left(e^{2\lambda w}\cosh\left(2\beta u  + 2 h^t\right) + e^{-2\lambda w}\cosh\left(2\beta v\right)\right)\right]\,.
    \]
    Since the only dependence on $v$ is through the quadratic and $\cosh(\cdot)$ component, $\td{F}(u,v,w)$ is even as a function of $v$. Now, fix $(u,w)$ and using the fact that $-\sinh^2(\cdot) + \cosh^2(\cdot) = 1$, we have
    \allowdisplaybreaks
    \begin{align*}
        &\quad \partial_{v,v}\td{F}(u,v,w) \\
        &= -2 + \E_{h^t}\left[\frac{4\beta^2 e^{-2\lambda w}\cosh(2\beta v)}{e^{2\lambda w}\cosh(2\beta u+2h^t) +e^{-2\lambda w}\cosh(2\beta v)}-\frac{4\beta^2 e^{-4\lambda w}\sinh^2(2\beta v)}{\left(e^{2\lambda w}\cosh(2\beta u+2h^t)+e^{-2\lambda w}\cosh(2\beta v)\right)^2}\right]\, \\ 
        &= -2 + 4\beta^2 \E_{h^t}\left[\frac{e^{-2\lambda w}\cosh(2\beta v)}{e^{2\lambda w}\cosh(2\beta u+2h^t)+e^{-2\lambda w}\cosh(2\beta v)} - \frac{\left(e^{-2\lambda w}\cosh(2\beta v)\right)^2-e^{-4\lambda w}}{\left(e^{2\lambda w}\cosh(2\beta u+2h^t)+e^{-2\lambda w}\cosh(2\beta v)\right)^2}\right] \\
        &= -2 + 4\beta^2\left(\E_{h^t}\left[\frac{e^{2\lambda w}\cosh(2(\beta u+h^t)) e^{-2\lambda w}\cosh(2\beta v)}{\left(e^{2\lambda w}\cosh(2\beta u+2h^t)+e^{-2\lambda w}\cosh(2\beta v)\right)^2}\right] + \E_{h^t}\left[\frac{e^{-4\lambda w}}{\left(e^{2\lambda w}\cosh(2\beta u+2h^t)+e^{-2\lambda w}\cosh(2\beta v)\right)^2}\right]\right) \\
        &\leq -2 + 4\beta^2\frac{1}{4} + 4\beta^2 \leq 5\beta^2 -2 <_{\beta <1/2}0\,,
    \end{align*}
    where we used the fact that,
    \allowdisplaybreaks
    \begin{align*}
        &\quad \frac{e^{-4\lambda w}}{\left(e^{2\lambda w}\cosh(2\beta u + 2 h^t) + e^{-2\lambda w}\cosh(2\beta v)\right)^2} \\
        &=\frac{e^{-4\lambda w}}{\cosh(2\beta u+2h^t)\left(2\cosh(2\beta v)+e^{4\lambda w} \cosh(2\beta u + 2h^t)\right)+e^{-4\lambda w} \cosh^2(2\beta v)} \\
        &\leq_{e^{(\cdot)}> 0, \cosh(\cdot)\geq 1} \frac{1}{\cosh^2(2\beta v)} \le 1\,. \qedhere
    \end{align*}
\end{prf}

The strict concavity and even-ness in $v$ immediately yield that every maximizer of $\td{F}(u,v,w)$ has $v= 0$. To argue uniqueness of the global maximizer, it suffices to show that $\nabla^2 \td{F}(u,0,w) \prec 0$ for some $0<\lambda < 7/16$ and $0< \beta <1/2$.

\begin{lemma}[Joint concavity of $\td{F}(u,0,w)$ in $(u,w)$ for $(\beta,\lambda) \in (0,1/2)\times(0,7/16)$]\label{lem:joint-concavity-basis-changed-coordinate-function}
    Fix $0 < \beta < 1/2$ and $0 < \lambda <7/16$. Then, for any $(u, w) \in \R^2$,
    \[
        \nabla^2_{u, w} \td{F}(u,0,w) \prec 0\,.
    \]
\end{lemma}
\begin{prf}
    Elementary algebra upon taking derivatives yields that
    \allowdisplaybreaks
    \begin{align*}
    -\nabla^2_{u, w} \td{F}(u,0,w) &= \begin{bmatrix}
        2 - 4\beta^2\E_{h^t}\left[\frac{e^{8\lambda w} + e^{4\lambda w}\cosh(2\beta u + 2h^t)}{\left(1 + e^{4\lambda w}\cosh(2\beta u + 2h^t)\right)^2}\right] & -8\beta\lambda\E_h\left[\frac{e^{4\lambda w}\sinh(2\beta u + 2h^t)}{\left(1 + e^{4\lambda w}\cosh(2\beta u + 2h^t)\right)^2}\right] \\
        -8\beta\lambda\E_{h^t}\left[\frac{e^{4\lambda w}\sinh(2\beta u + 2h^t)}{\left(1 + e^{4\lambda w}\cosh(2\beta u + 2h^t)\right)^2}\right] & 2\lambda - 16\lambda^2\E_h\left[\frac{e^{4\lambda w}\cosh(2\beta u + 2h^t)}{(1+e^{4\lambda w}\cosh(2\beta u + 2 h^t))^2}\right]\, 
        \end{bmatrix}\, \\
        &= \E_{h^t} \begin{bmatrix}
        2-\frac{4\beta^2 \left(e^{8\lambda w} + e^{4\lambda w}\cosh(2\beta u + 2h^t)\right)}{\left(1 + e^{4\lambda w}\cosh(2\beta u + 2h^t)\right)^2} & -\frac{8\beta\lambda e^{4\lambda w}\sinh(2\beta u + 2h^t)}{\left(1 + e^{4\lambda w}\cosh(2\beta u + 2h^t)\right)^2} \\
        - \frac{8\beta\lambda e^{4\lambda w}\sinh(2\beta u + 2h^t)}{\left(1 + e^{4\lambda w}\cosh(2\beta u + 2h^t)\right)^2} & 2\lambda -  \frac{16\lambda^2 e^{4\lambda w}\cosh(2\beta u + 2h^t)}{(1+e^{4\lambda w}\cosh(2\beta u + 2 h^t))^2}
        \end{bmatrix} .
    \end{align*}
    For any fixed $h^t(\omega), u, w$, let
    \[
    p=\frac{1}{1+e^{4\lambda w}\cosh(2\beta u + 2h^t(\omega))} \in (0,1), \quad x=\tanh(2\beta u + 2h^t(\omega)) \in (-1,1), 
    \]
    then simple calculation yields
    \[
    \begin{aligned}
        & 2-\frac{4\beta^2 \left(e^{8\lambda w} + e^{4\lambda w}\cosh(2\beta u + 2h^t)\right)}{\left(1 + e^{4\lambda w}\cosh(2\beta u + 2h^t)\right)^2} = 2 - 4\beta^2\left((1-p) - (1-p)^2 x^2\right), \\
        & -\frac{8\beta\lambda e^{4\lambda w}\sinh(2\beta u + 2h^t)}{\left(1 + e^{4\lambda w}\cosh(2\beta u + 2h^t)\right)^2} = -8\beta\lambda p (1-p)x ,\\
        & 2\lambda -  \frac{16\lambda^2 e^{4\lambda w}\cosh(2\beta u + 2h^t)}{(1+e^{4\lambda w}\cosh(2\beta u + 2 h^t))^2} = 2\lambda- 16\lambda^2 p(1-p).
    \end{aligned}
    \]
    Define a $2 \times 2 $ matrix as:
    \[
    B(h^t(\omega), u, w) = \begin{bmatrix}
    2 - 4\beta^2\left((1-p) - (1-p)^2 x^2\right) & -8\beta\lambda p (1-p)x \\
    -8\beta\lambda p (1-p)x & 2\lambda- 16\lambda^2 p(1-p) 
    \end{bmatrix} \in \R^{2 \times 2},
    \]
    then 
    \[
    -\nabla^2_{u, w} \td F(u, 0,w) = \E_{h^t}\left[B(h^t(\omega), u,w)\right] .
    \]
    We only need to show for any $h^t(\omega)$, $B(h^t(\omega), u, w)\succ 0$, then for any 2-dimensional nonzero vector $V \in \R^2$, 
    \[
    V^\top (-\nabla^2_{u, w} \td F(u, 0, w)) V = \E_{h^t} \underbrace{\left[V^\top B(h^t(\omega), u, w) V\right]}_{>0} >0 ,
    \]
    which implies $\nabla^2 \td F(u, 0, w) \prec 0$. Since $p>0$ and $(1-p)^2x^2\geq0$,
    \[
    (B(h^t(\omega), u,w)_{11} = 2- 4\beta^2\left(1-p-(1-p)^2x^2\right) >2 - 4\beta^2 >_{\beta <1/2} 0.
    \]
    Next,
    \[
    \begin{aligned}
        \det(B(h^t(\omega), u, w)) &= \left(2 - 4\beta^2(1-p) + 4\beta^2(1-p)^2 x^2\right) \left(2\lambda- 16\lambda^2 p(1-p)\right) - 64\beta^2\lambda^2p^2(1-p)^2 x^2 \\
        &= 4\lambda(1-2\beta^2(1-p))(1-8\lambda p(1-p))(1-x^2)+\left(4\lambda \left(1-2\beta^2(1-p) + 2\beta^2(1-p)^2\right)\right. \\
        & \qquad \qquad \left.\left(1-8\lambda p(1-p)\right)-64\beta^2\lambda^2p^2(1-p)^2\right) x^2 \\
        &= 4\lambda(1-2\beta^2(1-p))(1-8\lambda p(1-p))(1-x^2) + 4\lambda \left(1-2\left(\beta^2 + 4\lambda\right)p(1-p)\right)x^2 .
    \end{aligned}
    \]
    Both $1-x^2$ and $x^2 \geq 0$ and they will not be $0$ simultaneously, hence it remains to show
    \[
    \text{both} \quad (1-2\beta^2(1-p))(1-8\lambda p(1-p)) \quad \text{and} \quad \left(1-2\left(\beta^2 + 4\lambda\right)p(1-p)\right) >0.
    \]
    Since $p \in (0,1)$, and $p(1-p)\leq 1/4$, we have
    \[
    (1-2\beta^2(1-p))(1-8\lambda p(1-p)) >(1-2\beta^2)\left(1-8\lambda \frac{1}{4}\right) >_{\beta<1/2, \lambda<7/16}\left(1- \frac{1}{2}\right)\left(1-\frac{7}{8}\right)>0,
    \]
    also
    \[
    \left(1-2\left(\beta^2 + 4\lambda\right)p(1-p)\right) >_{\beta<1/2, \lambda<7/16} \left(1-2\left(\frac{1}{4}+\frac{7}{4}\right)\frac{1}{4}\right) = 0. \qedhere
    \]
\end{prf}

Based on the analysis above, $\td c$ is the unique global maximizer for $\td F(u, v, w)$. Next, we immediately derive some analytic properties of this function.

\paragraph{Negative-definiteness of $\nabla^2 \td{F}(\beta m^*,0,0)$ for $(\beta,\lambda) \in (0,1/2)\times(0,7/16)$}
We now show $\nabla^2 \td{F}(\td c)\prec 0 $ since we will apply the Laplace-approximation theorem to $\bar F(u, v,w)$ around the critical point shown in~\pref{prop:stationary-point-rate-function-variance-mgf}, and then control the remainder terms with arguments akin to~\cite[\S3]{kabluchko2024propagation}.

\begin{lemma}[Negative definiteness of the Hessian at the critical point]\label{cor:neg-def-hessian-ibp}
Let $0<\beta<1/2$ and $0<\lambda<7/16$. Set $\td F:\R^3\to\R$ as in~\pref{def:empirical-averages-deterministic-rate-function-u,v,w}, then
\[
\nabla^2 \td F(\td c) \prec 0.
\]

\end{lemma}

\begin{prf}
Since $\td F(u,v,w)$ is an even function of $v$, then for any $u, w$,
    \[
    \partial_v \td F(u,0,w)=0,\qquad \partial_{uv}\td F(u,0,w)=0,\qquad \partial_{wv}\td F(u,0,w)=0.
    \]
    Hence
    \[
    \nabla^2\td F(\td c)=
    \begin{pmatrix}
    \partial_{u,u}\td F(\td c)&0&\partial_{u,w}\td F(\td c)\\
    0&\partial_{v,v}\td F(\td c)&0\\
    \partial_{u,w}\td F(\td c)&0&\partial_{w,w}\td F(\td c)
    \end{pmatrix}.
    \]
    Therefore, for any nonzero vector $y:=(y_u, y_v, y_w)^\top \in \R^3$,
    \[
    y^\top \nabla^2 \td F(\td c) y = (y_u, y_w) \nabla_{u,w}^2 \td F(\td c) \binom{y_u}{y_w} + \partial_{v,v} \td{F}(\td c) y_v^2.
    \]
    By~\pref{lem:strict-concavity-f-transformed-basis}, 
    \[
    \partial_{v,v} \td{F}(\td c) = -2 + 4\beta^2 \E_{h^t}\left[\frac{1}{\cosh(2\beta^2 m^* + 2h^t) + 1}\right]\leq -2 + 4\beta^2 <_{\beta<1/2}0 .
    \]
    Besides, by~\pref{lem:joint-concavity-basis-changed-coordinate-function}, $ \nabla^2_{u,w}\td{F}(\td c)\prec 0$. Therefore for any nonzero vector $y$,
    \[
        y^\top \nabla^2 \td F(\td c)y <0 \implies \nabla^2\td F(\td c) \prec 0\,.
    \]
\end{prf}

\paragraph{Uniform bound on $\partial_{j,k,l} \bar F(\mathbf h; u, v, w)_{j ,k,l \in \{u,v,w\}}$ and $\partial_{j,k,l} \td F(u, v, w)_{j ,k,l \in \{u,v,w\}}$} To do a Taylor expansion with remainder around the critical point $\td c$ in a small region and use a Gaussian integral to evaluate the Hessian contribution, it is critical to uniformly bound the third-order derivatives and show that the term is negligible.

\begin{proposition}[Uniform bound for third-derivatives of $ \bar{F}$ and $\td{F}$]\label{prop:uniform-bound-f-third-derivative}
    Let $\bar{F}$ and $\td F$ be as defined in \pref{def:empirical-averages-deterministic-rate-function-u,v,w}, then for every choice of $j,k, l \in \{u, v, w\}$ and any $\beta, \lambda>0$,
    \[
        \left|\partial_{j,k, l} \bar F(\mathbf h; u, v, w)\right| \le C_{\beta, \lambda}< \infty\, ,
    \]
    and
    \[
    \left|\partial_{j,k, l} \td F(u, v, w)\right| \le C_{\beta, \lambda}< \infty\, ,
    \]
    where $C_{\beta, \lambda}$ is some constant depending on $\beta$ and $\lambda$.
\end{proposition}
\begin{prf}
Note that the linear and quadratic part of $\bar F$ do not contribute to the third derivatives, and only the logarithmic term contributes. Let $A(h^t;u,v,w)$ be as defined by~\pref{eq:function-of-u,v,w,h^t}, then
    \[
    A_i:=A(h_i;u,v,w) = e^{2\lambda w}\cosh\left( 2\beta u + 2h_i\right) + e^{-2\lambda w}\cosh\left(2\beta v\right),
    \]
    where $\{h_i\}_{i \in [n]}$ are fixed realization. Then
    \[
        \bar F(\mathbf h; u,v,w) = -\left(u^2 + v^2 + \lambda w^2\right) - 2\lambda w q^* + \frac{1}{n}\sum_{i=1}^n \ln\left(A_i(h_i;u,v,w)\right)\,.
    \]
    Hence for any $j, k, l\in\{u,v,w\}$,
    \begin{align*}
        &\partial_{j,k,l}\bar F(\mathbf h; u,v,w) = \frac{1}{n}\sum_{i=1}^n\partial_{j,k,l} \ln\left(A_i\right) = \frac{1}{n}\sum_{i=1}^n \partial_{j,k}\left(\frac{\partial_{l} A_i}{A_i}\right) = \frac{1}{n}\sum_{i=1}^n \partial_{j}\left(\frac{A_i\partial_{k,l}A_i - \partial_{l}A_i\partial_{k}A_i}{A_i^2}\right) \\
        &= \frac{1}{n}\sum_{i=1}^n \left(\frac{A_i^2\partial_j\left(A_i\partial_{k,l}A_i - \partial_kA_i\partial_l A_i\right) - 2A_i^2\partial_jA_i\partial_{k,l}A_i + 2A_i\partial_j A_i\partial_k A_i\partial_l A_i}{A_i^4}\right) \\
        &= \frac{1}{n}\sum_{i=1}^n \left(\frac{A_i^2\partial_{j}A_i\partial_{k,l}A_i + A_i^3\partial_{j,k,l}A_i - A_i^2\partial_{j,k}A_i\partial_l A_i - A_i^2\partial_{j,l}A_i\partial_k A_i - 2A_i^2\partial_jA_i\partial_{k,l}A_i + 2A_i\partial_j A_i\partial_k A_i\partial_l A_i}{A_i^4}\right) \\
        &= \frac{1}{n}\sum_{i=1}^n \left(\frac{1}{A_i^2}\left(-\partial_j A_i\partial_{k,l}A_i - \partial_{j,k}A_i\partial_l A_i -\partial_{jl}A_i \partial_k A_i\right) + \frac{\partial_{j,k,l}A_i}{A_i} + 2\frac{\partial_j A_i\partial_k A_i\partial_l A_i}{A_i^3}\right).
    \end{align*}
    Define $C_u=C_v=2\beta$, $C_w=2\lambda$, using the facts that $|\sinh(\cdot)| \le \cosh(\cdot)$ and $1 \le \cosh(\cdot)$ in conjunction with triangle inequality, we have $\forall r_1, \dots, r_m \in \{u, v, w\}$ and $m \in \{1,2,3\}$,
    \[
    \left|\partial_{r_1, \dots, r_m} A_i\right| \leq \left(\prod_{i=1}^m C_{r_i}\right) A_i .
    \]
    Therefore 
    \[
    |\partial_{j,k,l}\bar F(\mathbf h; u,v,w)|\leq 6C_j C_k C_l \leq 48 \max\{\beta^3, \beta^2\lambda, \beta\lambda^2, \lambda^3\}=48\max\{\beta ,\lambda\}^3:= C_{\beta, \lambda} < \infty
    \]
    The bound for $\left|\partial_{j,k,l}\td{F}(u,v,w)\right|$ follows straightforwardly by replacing the empirical average with $\E_{h^t}$ for the deterministic rate function and using the dominated convergence theorem to interchange the expectation and the derivative. \qedhere
\end{prf}

\subsection{Random Laplace approximation around the critical point.}\label{subsec:laplace-approximation-around-critical-point}

We first establish a lemma characterizing the almost-sure convergence of $\nabla^2\td{\Delta}(\mathbf h; \td c)$, which will be needed in the proof of the Taylor expansion component of the rigorous Laplace approximation (\pref{thm:whp-overlap-concentration-laplace-cwrf}).

\begin{lemma}[Almost-sure convergence of empirical Hessian to deterministic Hessian]\label{lem:as-convergence-hessian-rate-function}
For almost every realization $\mathbf h = \mathbf h^t(\omega)$ and any $\beta, \lambda >0$,
\[
\max_{a,b\in\{u,v,w\}}\left|\partial_{a,b}\td{\Delta}(\mathbf h; \td{c})\right|\longrightarrow 0.
\]
\end{lemma}

\begin{prf}
Let $A(h^t;u,v,w)$ be as defined by~\pref{eq:function-of-u,v,w,h^t}, and $a, b \in \{u, v, w\}$. Recall that $h_i^t\stackrel{\mathrm{i.i.d.}}{\sim}
\mathcal N\left(t,\beta^2 q^* + t\right)$ and $C_a, C_b \in \{2\beta, 2\lambda\}$. Since
\[
\left|\partial_{a,b}\ln A(h_i^t; \td{c})-\E_{h^t}\left[\partial_{a,b}\ln A(h^t;\td{c})\right]\right|\leq 4 C_a C_b,
\]
then by Kolmogorov's strong law of large numbers,
\[
\frac{1}{n}\sum_{i=1}^n \left(\partial_{a,b}\ln A(h_i^t; \td{c})-\E_{h^t}\left[\partial_{a,b}\ln A(h^t;\td{c})\right]\right)\xrightarrow{\mathrm{a.s.}}\E_{\mathbf h^t}\left[\partial_{a,b}\ln A(h_i^t; \td{c})-\E_{h^t}\!\left[\partial_{a,b}\ln A(h^t;\td{c})\right]\right]=0,
\]
which implies $|\partial_{a,b}\td{\Delta}(\mathbf h^t; \td{c})| \to 0$ almost surely.

It remains to show that the convergence holds simultaneously for all
entries of the Hessian. For each $a,b\in\{u,v,w\}$, define
\[
\Omega_{a,b}:=\left\{\omega:\lim_{n\to\infty}\left|\partial_{a,b}\td{\Delta}\left(\mathbf h^t(\omega); \td{c}\right)\right|=0\right\}.
\]
The preceding argument shows that $\P(\Omega_{a,b})=1$ for every $a,b\in\{u,v,w\}$. Since there are only finitely many such
pairs, we have
\allowdisplaybreaks
\begin{align*}
    \P\left(\omega: \lim_{n \to \infty} \max_{a, b \in \{u, v,w\}} \left|\partial_{a, b} \td \Delta\left(\mathbf h^t(\omega); \td{c}\right)\right| = 0\right) &= \P\left(\omega: \max_{a, b\in \{u, v,w\}} \lim_{n \to \infty}\left|\partial_{a, b} \td\Delta\left(\mathbf h^t(\omega); \td{c}\right)\right| = 0\right) \\
    &= \P\left(\bigcap_{a,b\in\{u,v,w\}} \Omega_{a, b}\right) = 1-\P\left(\bigcup_{a,b\in\{u,v,w\}} \Omega^c_{a, b}\right) \\
    &\geq 1- \sum_{a, b \in \{u, v,w\}} \left(1- \P\left(\Omega_{a, b}\right)\right) = 1. 
\end{align*}

Therefore, for almost every realization $\mathbf h = \mathbf h^t(\omega)$,
\[
\max_{a,b \in \{u, v, w\}} \left|\partial_{a, b} \td \Delta(\mathbf h; \td c)\right| \longrightarrow 0 . \qedhere
\]
\end{prf}

Next, we bound $\left\langle e^{\lambda n\left(R_{\sigma,\tau}-q^*\right)^2}\right\rangle$ for sufficient large $n$.

\begin{theorem}[Laplace approximation for $3$-dimensional random rate functions with sufficient regularity]\label{thm:whp-overlap-concentration-laplace-cwrf}
    Fix $(\beta,\lambda) \in (0,1/2)\times (0,7/16)$. Let $(m^*,q^*)$ be the unique solutions to~\pref{eq:fixpoint-equations}. Then, for almost every realization $\mathbf h = \mathbf h^t(\omega)$, sufficiently large $n$, and some constant $C_{1, 1}(\beta,\lambda) < \infty$,
    \[
        \an{e^{\lambda n (R_{\sigma,\tau}-q^*)^2}}_\mathbf h \le C_{1, 1}(\beta,\lambda)e^{n\Gamma_1(\mathbf h; \td c)/2}\, ,
    \]
    where
    \[
        \Gamma_1(\mathbf h; \td c) := \nabla\td\Delta(\mathbf h; \td c)^\top \left(-\nabla^2 \td F(\td c) - \varepsilon I_3\right)^{-1}\nabla\td\Delta(\mathbf h; \td c),
    \]
    for some sufficiently small $\eps > 0$.
\end{theorem}
\begin{prf}

    Recall that $F$ is defined by~\pref{eq:empirical-averages-function}, $\bar F$ and $\td F$ is defined in~\pref{def:empirical-averages-deterministic-rate-function-u,v,w}, and $\td \Delta$ is defined in~\pref{def:laplace-deviation-function-u,v,w}. 
    
    Throughout the remainder of the proof, we fix a realization $\mathbf h=\mathbf h^t(\omega)$. To simplify notation, we suppress the dependence of $\bar F$ and $\td\Delta$ on $\mathbf h$, unless we wish to emphasize the chosen realization or replace $\mathbf h$ by the random field
$\mathbf h^t$.

    \ppart{Decomposition of neighborhoods} Let $z:=(z_u, z_v, z_w)=(u, v, w)- \td c$. We decompose the integral into $3$ neighborhoods:
    \[
        \left\{\norm{z}_2 \le \delta \right\} \cup \left\{\delta < \norm{z}_2 \le R\right\} \cup \left\{\norm{z}_2 > R\right\}\,.
    \]
    The contribution of the first region, which will be dominant, is estimated via a Gaussian integral after doing a Taylor expansion with remainder. The contribution from the second region is controlled almost-surely by combining the negativity of the maximum of $\td \Delta$ with tail bounds for the supremum of sub-Gaussian processes and the Borel-Cantelli lemma. The last term is shown to be small via coercivity. Therefore, we have
    \[
    \begin{aligned}
        \int_{\R^3}e^{nF(x,y,w)}dx\,dy\,dw &= \left|\det \frac{\partial(x,y,w)}{\partial(u,v,w)}\right|\int_{\R^3}e^{n \bar F(z + \td c)}dz \\
        &= 2\underbrace{\int_{\norm{z}_2 \le \delta}e^{n \bar F(z+ \td c)}dz }_{\mathsf{I}} + 2\underbrace{\int_{\delta \leq \norm{z}_2 \le R}e^{n\bar F(z+ \td c)}dz}_{\mathsf{II}} + 2\underbrace{\int_{\norm{z}_2 \geq R}e^{n\bar F(z+ \td c)}dz}_{\mathsf{III}}\,.
    \end{aligned}
    \]
    The lower bound on the denominator $\int_{\R^2} e^{nF(x,y,0)}dx\,dy$ follows rather easily once the numerator is estimated with similar arguments, and is delegated towards the end of the proof.

    \ppart{\texorpdfstring{$\mathsf{I}$}{I}: Bound via Taylor expansion with remainder}  We do Taylor's expansion to $\bar F(z+\td c)$ around critical point $\td c = (\beta m^*, 0, 0)$. Since $\bar F=\td F+\td\Delta$ and $\nabla\td F(\td c)=0$, we have,
    \[
    \begin{aligned}
        \bar F(z+\td c) &= \bar{F}(\td{c})+\nabla\td \Delta(\td c)^{\top} z +\frac{1}{2} z^{\top} \nabla^2 \bar{F}(\td{c}) z+\frac{1}{6} \nabla^3 \bar{F}(\kappa)[z, z, z] \\
        &= \bar{F}(\td{c})+\nabla\td \Delta(\td c)^{\top} z +\frac{1}{2} z^{\top} \nabla^2 \td{F}(\td{c}) z + \frac{1}{2} z^{\top} \nabla^2 \td{\Delta}(\td{c}) z +\frac{1}{6} \nabla^3 \bar{F}(\kappa)[z, z, z],
    \end{aligned}
    \]
    for some point $\kappa$ between $\td{c}$ and $\td{c} + z$.
    
    By~\pref{cor:neg-def-hessian-ibp}, we have 
  \[
  \lambda_*:=\lambda_{\min}\left(-\nabla^2\td F(\td c)\right)>0.
  \]
  Choose any $\varepsilon\in(0,\lambda_*)$, for example, $\varepsilon=\lambda_*/2$ and define
\[
A_{\mathrm{num}}:=-\nabla^2\td F(\td c)-\varepsilon I_3 \succ0.
\]
To control the third-order remainder in the Taylor expansion, we use the uniform bound for third-derivative of $\bar F$ established in~\pref{prop:uniform-bound-f-third-derivative}. Viewing $\nabla^3\bar F(\kappa)$ as a third-order tensor and applying the tensor version of the Cauchy--Schwarz inequality, for any $j ,k ,l \in \{u,v,w\}$,
\[
\begin{aligned}
\left|\nabla^3\bar F(\kappa)[z,z,z]\right|&=\left|\sum_{j,k,l \in \{u, v,w\}}\partial_{j,k,l}\bar F(\kappa)\,z_jz_kz_l\right|\\
&\leq\left(\sum_{j,k,l \in \{u, v,w\}}\left|\partial_{j,k,l}\bar F(\kappa)\right|^2\right)^{1/2}\left(\sum_{j,k,l \in \{u, v,w\}}z_j^2z_k^2z_l^2\right)^{1/2}\\
&\leq \left(\sum_{j,k,l \in \{u, v,w\}} C_{\beta, \lambda}^2\right)^{1/2}\|z\|_2^3 = 3\sqrt{3} C_{\beta, \lambda} \|z\|_2^3.\end{aligned}
\]
Then, there exists a constant $M := 3\sqrt{3}C_{\beta, \lambda} <\infty$ such that $\left|\nabla^3\bar F(\kappa)[z,z,z]\right|\leq M\|z\|_2^3$. Choose $\delta>0$ sufficiently small so that $M\delta/6 \leq \varepsilon/4$, for example, we may take $\delta = 3\varepsilon/(2M)$. Then, whenever $\|z\|_2\leq\delta$,
\[
\frac16 \left|\nabla ^3\bar F(\kappa)[z,z,z]\right|\leq\frac{M}{6}\|z\|_2^3\leq\frac{\varepsilon}{4}\|z\|_2^2.
\]

By~\pref{lem:as-convergence-hessian-rate-function}, for almost every $\omega$, there exists $n_2(\omega)>0$ such that,
as $n\geq n_2(\omega)$,
\[
\max_{a,b\in\{u,v,w\}}\left|\partial_{a,b}\td\Delta(\mathbf h; \td c)\right|\leq\frac{\varepsilon}{6}.
\]
Consequently,
\[
\begin{aligned}
\left|z^\top\nabla^2\td\Delta(\td c)z\right|&\leq_{\text{Cauchy--Schwarz}}\|z\|_2\left\|\nabla^2\td\Delta(\td c)z\right\|_2 \leq\left\|\nabla^2\td\Delta(\td c)\right\|_{\mathrm{op}}\|z\|_2^2\\
&\leq\left\|\nabla^2\td\Delta(\td c)\right\|_{\mathrm F}\|z\|_2^2=\left(\sum_{a,b\in\{u,v,w\}}\left|\partial_{a,b}\td\Delta(\td c)\right|^2\right)^{1/2}\|z\|_2^2\\
&\leq3\max_{a,b\in\{u,v,w\}}\left|\partial_{a,b}\td\Delta( \td c)\right|\|z\|_2^2\leq\frac{\varepsilon}{2}\|z\|_2^2.
\end{aligned}
\]
Therefore, when $\|z\|_2\leq\delta$ and $n\geq n_2(\omega)$,
\[
\begin{aligned}
\bar F(\td c+z)&=\bar F(\td c)+\nabla\td\Delta(\td c)^\top z+\frac12z^\top\nabla^2\td F(\td c)z+\frac12z^\top\nabla^2\td\Delta(\td c)z+\frac16\nabla^3\bar F(\kappa)[z,z,z]\\
&\leq\bar F(\td c)+\nabla\td\Delta(\td c)^\top z-\frac12z^\top\left(-\nabla^2\td F(\td c)\right)z+
\frac{\varepsilon}{4}\|z\|_2^2+\frac{\varepsilon}{4}\|z\|_2^2 \\
&\leq\bar F(\td c)+\nabla\td\Delta(\td c)^\top z-\frac12z^\top A_{\mathrm{num}}z.
\end{aligned}
\]
Then for almost every $\omega$, as $n\geq n_2(\omega)$, the contribution from the first region therefore satisfies
\allowdisplaybreaks
\begin{align*}
\mathsf{I} &=\int_{\|z\|_2\leq\delta}\exp\left\{n\bar F(\mathbf h;\td c+z)\right\}\,dz \leq e^{n\bar F(\mathbf h;\td c)}\int_{\|z\|_2\leq\delta}\exp\left\{n\left(\nabla\td\Delta(\mathbf h; \td c)^\top z-\frac12z^\top A_{\mathrm{num}}z\right)\right\}\,dz \\ 
&=e^{n\bar F(\mathbf h;\td c)}\int_{\|z\|_2\leq\delta}\exp\left\{-\frac n2\left(z-A_{\mathrm{num}}^{-1}\nabla\td\Delta(\mathbf h; \td c)\right)^\top A_{\mathrm{num}} \left(z-A_{\mathrm{num}}^{-1}\nabla\td\Delta(\mathbf h; \td c)\right)+\frac n2\nabla\td\Delta(\mathbf h; \td c)^\top A_{\mathrm{num}}^{-1}\nabla\td\Delta(\mathbf h; \td c)\right\}\,dz \\
&\leq \exp\left\{n\bar F(\mathbf h;\td c)+\frac n2\Gamma_1(\mathbf h; \td c)\right\} \int_{\R^3} \exp\left\{-\frac{n}{2}y^\top A_{\mathrm{num}} y\right\} \,dy = \left(\frac{2\pi}{n}\right)^{3/2}\frac{1}{\sqrt{\det A_{\mathrm{num}}}}\exp\left\{n\bar F(\mathbf h;\td c)+\frac n2\Gamma_1(\mathbf h; \td c)\right\}.
\end{align*}
    
    \ppart{\texorpdfstring{$\mathsf{II}$}{II}: Upper bound via sub-Gaussian concentration and Borel-Cantelli} The second integral is handled by proving that $\td F(\td c)-\sup _{\delta \leq\|z\|_2 \leq R}\td F(\td c+z)$ can be controlled by combining the ``gap'' in the function value far away from the global maximizer with tail bounds for sub-Gaussian processes.
    
    Define
\[
\gamma_{\delta,R}
:=
\td F(\td c)
-
\sup_{\delta\leq\norm{z}_2\leq R}
\td F(\td c+z)
>0.
\]
Indeed, since $\td F$ is continuous, it reaches its maximum on the compact set $\left\{ z\in\R^3: \delta\leq\norm{z}_2\leq R \right\}$. Since $\td c$ is the unique global maximizer of $\td F$, this maximum is strictly smaller than $\td F(\td c)$. Therefore, $\gamma_{\delta,R}>0$.

Let $A(h^t;z)$ be as defined in~\pref{eq:function-of-u,v,w,h^t}, and recall $h_1^t,\ldots,h_n^t \stackrel{\mathrm{i.i.d.}}{\sim} \mathcal N\left(t,\beta^2 q^* + t\right)$. We next show that $\td\Delta(\mathbf h;\cdot)$ is Lipschitz in $z$, and $A(\cdot;z)$ is Lipschitz in $h$ for any fixed $z$. Define $L:=2\sqrt{2\beta^2+\lambda^2}$. Mean-value theorem gives for each $i\in[n]$,
\[
\left|\ln A(h_i;z) - \ln A(h_i;z^\prime)\right| \leq \left\|\nabla_{z} \ln A(h_i;z)\right\|_2  \norm{z-z^\prime}_2 \leq \sqrt{(2\beta)^2+(2\beta)^2+(2\lambda)^2} \norm{z-z^\prime}_2 = L \norm{z-z^\prime}_2.
\]
Hence,
\[
\left|\td\Delta(\mathbf h;z)-\td\Delta(\mathbf h;z^\prime)\right|\leq\frac1n\sum_{i=1}^n\left|\ln A(h_i;z)-\ln A(h_i;z^\prime)\right|+\E_{h^t}\left[\left|\ln A(h^t;z)-\ln A(h^t;z^\prime)\right|\right]\leq 2L\norm{z-z^\prime}_2.
\]
Thus, $\td\Delta(\mathbf h;\cdot)$ is $2L$-Lipschitz. Fix $z$, then
\[
\left|\partial_h\ln A(h;z)\right|=\left|\frac{2e^{2\lambda w}\sinh(2\beta u+2h)}{e^{2\lambda w}\cosh(2\beta u+2h)+e^{-2\lambda w}\cosh(2\beta v)}\right|\leq 2.
\]
Therefore, $\ln A(\cdot; z)$ is 2-Lipschitz. As $\beta^2 q^* + t \neq 0$, the concentration inequality for Lipschitz
functions of Gaussian random variables
\cite[Theorem~5.3.1]{vershynin2018high} then implies that, for every $a>0$,
\[
\P\left(\left|\ln A(h_i^t;z)-\E_{h^t}\left[\ln A(h^t;z)\right]\right|\geq a\right)\leq2\exp\left\{-\frac{a^2}{8\left(\beta^2 q^* + t\right)}\right\}.
\]
Thus, each random variable $\ln A(h_i^t;z)-\E_{h^t}\left[\ln A(h^t;z)\right]$ is centered, independent, and sub-Gaussian. Since $\E_{\mathbf h^t} \left[\td\Delta(\mathbf h^t;z)\right]=0$, the sub-Gaussian tail bound for sums of independent random variables \cite[Corollary~2.11]{boucheron2012} gives
\[
\P\left(\left|\td\Delta(\mathbf h^t;z)\right|\geq a\right)\leq2\exp\left\{-\frac{a^2n}{8\left(\beta^2 q^* + t\right)}\right\}.
\]
We now consider the ball centered at $\td c$, $B_R(\td c):=\left\{\td c+z:\norm{z}_2\leq R\right\}$. To control the supremum at level $\gamma_{\delta,R}/4$, choose the radius of the net to be $\eta:=\gamma_{\delta,R}/(16L)$. There exists an $\eta$-net $\mathcal N_\eta$ of $B_R(\td c)$
satisfying
\[
\left|\mathcal N_\eta\right|\leq\left(1+\frac{2R}{\eta}\right)^3=\left(1+\frac{32RL}{\gamma_{\delta,R}}\right)^3.
\]
By the property of $\eta$-net, for every $z\in B_R(\td c)$, there exists $z^\prime\in\mathcal N_\eta$ such that $\norm{z-z^\prime}_2\leq\eta$. Since $\td\Delta(\mathbf h^t;\cdot)$ is $2L$-Lipschitz,
\[
\left|\td\Delta(\mathbf h^t;z)\right|\leq\left|\td\Delta(\mathbf h^t;z^\prime)\right|+2L\eta=\left|\td\Delta(\mathbf h^t;z^\prime)\right|+\frac{\gamma_{\delta,R}}{8}.
\]
Consequently, the following inclusion of events holds:
\[
\left\{\omega:\sup_{z\in B_R(\td c)}\left|\td\Delta(\mathbf h^t(\omega);z)\right|>\frac{\gamma_{\delta,R}}{4}\right\}\subseteq\bigcup_{z^\prime\in\mathcal N_\eta}\left\{\omega:\left|\td\Delta(\mathbf h^t(\omega);z^\prime)\right|>\frac{\gamma_{\delta,R}}{8}\right\}.
\]
Therefore, by the union bound and the preceding result of sub-Gaussianity,
\allowdisplaybreaks
\begin{align}\label{eq:summation-of-probabilities}
\sum_{n=1}^\infty\P\left(\sup_{z\in B_R(\td c)}\left|\td\Delta(\mathbf h^t;z)\right|>\frac{\gamma_{\delta,R}}{4}\right)&\leq\sum_{n=1}^\infty\sum_{z^\prime\in\mathcal N_\eta}\P\left(\left|\td\Delta(\mathbf h^t;z^\prime)\right|\geq\frac{\gamma_{\delta,R}}{8}\right)\\
&\leq2\sum_{n=1}^\infty\left(1+\frac{32RL}{\gamma_{\delta,R}}\right)^3\exp\left\{-\frac{n\gamma_{\delta,R}^2}{512\left(\beta^2 q^* + t\right)}\right\}<\infty.
\end{align} 
When $\beta^2 q^* + t = 0$, $\P\left(\left|\td\Delta(\mathbf h^t;z^\prime)\right|\geq a\right)=0$. This does not change the result in \pref{eq:summation-of-probabilities}.

By the Borel--Cantelli lemma, for almost every $\omega$, there exists a finite random integer $n_3(\omega)>0$ such that, for all $n\geq n_3(\omega)$,
\[
\sup_{z\in B_R(\td c)}\left|\td\Delta(\mathbf h^t(\omega);z)\right|\leq\frac{\gamma_{\delta,R}}{4}.
\]
Equivalently, for almost every realization $\mathbf h=\mathbf h^t(\omega)$, and all $n\geq n_3(\omega)$,
\[
\sup_{z\in B_R(\td c)}\left|\td\Delta(\mathbf h;z)\right|\leq\frac{\gamma_{\delta,R}}{4}.
\]

Finally, we control the contribution from the second region. For every
$z$ satisfying $\delta\leq\norm{z}_2\leq R$,
\[
\bar F(\td c+z)=\td F(\td c+z)+\td\Delta(\td c+z)\leq\td F(\td c)-\gamma_{\delta,R}+\sup_{z\in B_R(\td c)}\left|\td\Delta(z)\right|.
\]
On the other hand,
\[
\bar F(\td c)=\td F(\td c)+\td\Delta(\td c)\geq\td F(\td c)-\sup_{z\in B_R(\td c)}\left|\td\Delta(z)\right|.
\]
Combining the preceding two inequalities, for all $n\geq n_3(\omega)$,
\[
\bar F(\td c+z)\leq\bar F(\td c)-\gamma_{\delta,R}+2\sup_{z\in B_R(\td c)}\left|\td\Delta(z)\right|\leq\bar F(\td c)-\frac{\gamma_{\delta,R}}{2}.
\]
Therefore, for any $R>\delta$($\delta$ is chosen in bounding part $\mathrm I$), almost every realization $\mathbf h=\mathbf h^t(\omega)$, and every $n\geq n_3(\omega)$,
\[
\mathrm{II}=\int_{\delta\leq\norm{z}_2\leq R}e^{n\bar F(\mathbf h; \td c+z)}\,dz\leq\operatorname{Vol}\left(B_R(\td c)\right)\exp\left\{n\bar F(\mathbf h; \td c)-\frac{n\gamma_{\delta,R}}{2}\right\}=\frac{4\pi R^3}{3}\exp\left\{n\bar F(\mathbf h; \td c)-\frac{n\gamma_{\delta,R}}{2}\right\}.
\]
    
\ppart{\texorpdfstring{$\mathsf{III}$}{III}: Gaussian upper bound via coercivity} Recall the upper bound for $\bar F(u,v,w)$ established in~\pref{prop:f-infinite-limits}:
\[
\bar F(u,v,w)\leq-\frac{\lambda}{2}\left(u^2+v^2+w^2\right)+\ln 2+4\beta^2+2\lambda(1+q^*)^2+2\left(\frac1n\sum_{i=1}^n|h_i|\right).
\]
We have the following inequality: $\norm{\td c+z}_2^2 \geq \frac12\norm{z}_2^2-\norm{\td c}_2^2$. Indeed,
\allowdisplaybreaks
\begin{align*}
\norm{\td c+z}_2^2&\geq\norm{z}_2^2+\norm{\td c}_2^2-2\left|\left\langle\td c,z\right\rangle\right|=\norm{z}_2^2+\norm{\td c}_2^2-2\left|\left\langle\sqrt{2}\,\td c,\frac{z}{\sqrt{2}}\right\rangle\right|\\
&\geq_{\text{Cauchy–Schwarz}}\norm{z}_2^2+\norm{\td c}_2^2-\left(2\norm{\td c}_2^2+\frac12\norm{z}_2^2\right)=\frac12\norm{z}_2^2-\norm{\td c}_2^2.
\end{align*}
Substituting this inequality into the preceding upper bound and using $(u,v,w)=\td c+z$, we obtain
\allowdisplaybreaks
\begin{align*}
\bar F(\td c+z)&\leq-\frac{\lambda}{4}\norm{z}_2^2+\frac{\lambda}{2}\norm{\td c}_2^2+\ln 2+4\beta^2+2\lambda(1+q^*)^2+2\left(\frac1n\sum_{i=1}^n|h_i|\right)\\
&=-\frac{\lambda}{4}\norm{z}_2^2+\frac{\lambda}{2}\beta^2(m^*)^2+\ln 2+4\beta^2+2\lambda(1+q^*)^2+2\left(\frac1n\sum_{i=1}^n|h_i|\right).
\end{align*}
As in the proof of~\pref{prop:f-infinite-limits}, the strong law of large numbers implies that, for almost every fixed realization $\mathbf h = \mathbf h^t(\omega)$, there exists $n_1(\omega)$ such that for every $n\geq n_1(\omega)$, $\frac1n\sum_{i=1}^n|h_i|\leq\E_{\mathbf h^t}|h_1^t|+1$. Define
\[
C_0(\beta, \lambda):=2\left(\E_{\mathbf h^t}|h_1^t|+1\right)+\frac{\lambda}{2}\beta^2(m^*)^2+\ln 2+4\beta^2+2\lambda(1+q^*)^2,
\]
then for every $n\geq n_1(\omega)$, $\bar F(\mathbf h; \td c+z)\leq-\frac{\lambda}{4}\norm{z}_2^2+C_0(\beta, \lambda)$. Choose $R>\delta$ sufficiently large so that
\[
\gamma_R:=\frac{\lambda R^2}{4}-C_0(\beta, \lambda)>0.
\]
Then for almost every realization $\mathbf h= \mathbf h^t(\omega)$, as $n\geq n_1(\omega)$,
\allowdisplaybreaks
\begin{align*}
\mathrm{III}&=\int_{\norm{z}_2>R}e^{n\bar F(\mathbf h; \td c+z)}\,dz\leq e^{nC_0(\beta, \lambda)}\int_{\norm{z}_2>R}\exp\left\{-\frac{n\lambda}{4}\norm{z}_2^2\right\}\,dz\\
&=4\pi e^{nC_0(\beta, \lambda)}\int_R^\infty r^2\exp\left\{-\frac{n\lambda}{4}r^2\right\}\,dr \leq 4\pi e^{nC_0(\beta, \lambda)} e^{-\frac{n\lambda R^2}{4}}\left(\frac{2R}{n \lambda}+\frac{4}{n^2\lambda^2R}\right) \\
&= 4\pi e^{-\gamma_R n} \left(\frac{2R}{n \lambda}+\frac{4}{n^2\lambda^2R}\right)
\end{align*}

\ppart{Lower bounding the denominator} 
    
First, we transform the integral to the $(u,v)$-basis. Define $\td c_0 = (\beta m^*,0)$, $z_0:=(z_u,z_v)$, and
\[
\bar F_0(\mathbf h;u,v):=\bar F(\mathbf h;u,v,0), \quad\td F_0(u,v):=\td F(u,v,0), \quad \td\Delta_0(\mathbf h;u,v):=\td\Delta(\mathbf h;u,v,0).
\]
Then
\[
\int_{\R^2}e^{nF(\mathbf h;x,y,0)}\,dx\,dy=\left|\det\frac{\partial(x,y)}{\partial(u,v)}\right|\int_{\R^2}e^{n\bar F_0(\mathbf h;z_0+\td c_0)}\,dz_0=2\int_{\R^2}e^{n\bar F_0(\mathbf h;z_0+\td c_0)}\,dz_0.
\]
Similar to bounding integral $\mathrm I$, we do a Taylor expansion to $\bar F_0$ around the critical point $\td c_0$, and control the Hessian and third-order remainder terms. Since $\bar F_0=\td F_0+\td\Delta_0$ and $\nabla\td F_0(\td c_0)=0$, Taylor's theorem with remainder gives
\[
\begin{aligned}
\bar F_0(z_0+\td c_0)=&\bar F_0(\td c_0)+\nabla\td\Delta_0(\td c_0)^\top z_0+\frac12z_0^\top\nabla^2\td F_0(\td c_0)z_0\\
&\quad +\frac12z_0^\top\nabla^2\td\Delta_0(\td c_0)z_0+\frac16\nabla^3\bar F_0(\kappa_0)[z_0,z_0,z_0],
\end{aligned}
\]
where $\kappa_0$ lies between $\td c_0$ and $\td c_0+z_0$. By~\pref{cor:neg-def-hessian-ibp}, $-\nabla^2\td F_0(\td c_0)\succ0 $, then we choose the same $\varepsilon=\lambda_{\min}\left(-\nabla^2\td F(\td c)\right)/2$
as in part $\mathrm I$, and define
\[
A_{\mathrm{den}}:=-\nabla^2\td F_0(\td c_0)+\varepsilon I_2\succ0.
\]

For the third-order remainder, by \pref{prop:uniform-bound-f-third-derivative},
\allowdisplaybreaks
\begin{align*}
\left|\nabla^3\bar F_0(\kappa_0)[z_0,z_0,z_0]\right|&\leq\left(\sum_{j,k,l\in\{u,v\}}\left|\partial_{j,k,l}\bar F(\kappa_0,0)\right|^2\right)^{1/2}\left(\sum_{j,k,l\in\{u,v\}}z_j^2z_k^2z_l^2\right)^{1/2}\\
&\leq\left(\sum_{j,k,l\in\{u,v\}}C_{\beta, \lambda}^2\right)^{1/2}\|z_0\|_2^3=2\sqrt{2}\,C_{\beta, \lambda}\|z_0\|_2^3.
\end{align*}
For convenience, we use the same constant as in part $\mathrm I$, $M=3\sqrt{3}\,C_{\beta, \lambda}$, so that
\[
\left|\nabla^3\bar F_0(\kappa_0)[z_0,z_0,z_0]\right|\leq M\|z_0\|_2^3.
\]
We also choose the same radius $\delta=3\varepsilon/(2M)$, then whenever $\|z_0\|_2 \leq\delta$,
\[
\frac16\left|\nabla^3\bar F_0(\kappa_0)[z_0,z_0,z_0]\right|\leq\frac{M}{6}\|z_0\|_2^3\leq\frac{\varepsilon}{4}\|z_0\|_2^2.
\]
Recall the random integer $n_2(\omega)$ chosen in part $\mathrm{I}$. By~\pref{lem:as-convergence-hessian-rate-function}, for almost every $\omega$ and every $n\geq n_2(\omega)$,
\[
\max_{a,b\in\{u,v\}}\left|\partial_{a,b}\td\Delta_0(\mathbf h;\td c_0)\right|=\max_{a,b\in\{u,v\}}\left|\partial_{a,b}\td\Delta(\mathbf h;\td c)\right|\leq\max_{a,b\in\{u,v,w\}}\left|\partial_{a,b}\td\Delta(\mathbf h;\td c)\right|\leq\frac{\varepsilon}{6}.
\]
Therefore,
\[
\left|z_0^\top\nabla^2\td\Delta_0(\td c_0)z_0\right|\leq\left\|\nabla^2\td\Delta_0(\td c_0)\right\|_{\mathrm F}\|z_0\|_2^2\leq2\max_{a,b\in\{u,v\}}\left|\partial_{a,b}\td\Delta_0(\td c_0)\right|\|z_0\|_2^2\leq\frac{\varepsilon}{3}\|z_0\|_2^2\leq\frac{\varepsilon}{2}\|z_0\|_2^2.
\]
Therefore, whenever $\|z_0\|_2\leq\delta$ and $n\geq n_2(\omega)$,
\[
\begin{aligned}
\bar F_0(z_0+\td c_0)&\geq\bar F_0(\td c_0)+\nabla\td\Delta_0(\td c_0)^\top z_0+\frac12z_0^\top\nabla^2\td F_0(\td c_0)z_0-\frac{\varepsilon}{2}\|z_0\|_2^2\\
&=\bar F_0(\td c_0)+\left(\nabla\td\Delta_0(\td c_0)\right)^\top z_0-\frac12z_0^\top A_{\mathrm{den}}z_0.
\end{aligned}
\]
For any realization $\mathbf h$, we define $\Gamma_2(\mathbf h; \td c_0) := \nabla\td\Delta_0(\mathbf h; \td c_0)^\top(A_{\text{den}})^{-1} \nabla\td\Delta_0(\mathbf h; \td c_0)$. Restricting $z_0$ to the ball of radius $\delta$, we obtain
\allowdisplaybreaks
\begin{align*}
\int_{\mathbb R^2}e^{n\bar F_0(z_0+\td c_0)}\,dz_0&\geq \int_{\|z_0\|_2\leq\delta} e^{n\bar F_0(z_0+\td c_0)} \,dz_0\geq e^{n\bar F_0(\td c_0)} \int_{\|z_0\|_2\leq\delta}\exp\left\{n\left(\left(\nabla\td\Delta_0(\td c_0)\right)^\top z_0-\frac12z_0^\top A_{\mathrm{den}}z_0\right)\right\}\,dz_0 \\& = e^{n\bar F_0(\td c_0)+\frac n2\Gamma_2(\td c_0)} \int_{\|z_0\|_2\leq\delta}\exp\left\{-\frac n2\left(z_0-A_{\mathrm{den}}^{-1}\nabla\td\Delta_0(\td c_0)\right)^\top A_{\mathrm{den}}\left(z_0-A_{\mathrm{den}}^{-1}\nabla\td\Delta_0(\td c_0)\right)\right\}\,dz_0.
\end{align*}
By definition of $\td \Delta_0$,
\[
\nabla\td\Delta_0(\mathbf h^t;\td c_0)=
\left(\frac1n\sum_{i=1}^n\left(\partial_u\ln A(h_i^t;\td c)-\E_{h^t}\left[\partial_u\ln A(h^t;\td c)\right]\right),\frac1n\sum_{i=1}^n\left(\partial_v\ln A(h_i^t;\td c)-\E_{h^t}\left[\partial_v\ln A(h^t;\td c)\right]\right)\right).
\]
Since $A_{\mathrm{den}}^{-1}$ is a fixed matrix, the strong law of large numbers gives, for almost every realization $\mathbf h=\mathbf h^t(\omega)$,
\[
A_{\mathrm{den}}^{-1}\nabla\td\Delta_0(\mathbf h;\td c_0)\longrightarrow 0.
\]
Therefore, for almost every $\omega$, there exists $n_4(\omega)>0$ such that, for every $n\geq n_4(\omega)$,
\[
\left\|A_{\mathrm{den}}^{-1}\nabla\td\Delta_0(\mathbf h;\td c_0)\right\|_2\leq\frac{\delta}{2}.
\]
Let $y:=z_0-A_{\mathrm{den}}^{-1}\nabla\td\Delta_0(\mathbf h;\td c_0)$, when $n\geq n_4(\omega)$ and $\|y\|_2\leq\delta/2$,
\[
\left\|y+A_{\mathrm{den}}^{-1}\nabla\td\Delta_0(\mathbf h;\td c_0)\right\|_2\leq\|y\|_2+\left\|A_{\mathrm{den}}^{-1}\nabla\td\Delta_0(\mathbf h;\td c_0)\right\|_2\leq\delta,
\]
which implies
\[
\left\{y:\|y\|_2\leq\frac{\delta}{2}\right\}\subseteq\left\{y:\left\|y+A_{\mathrm{den}}^{-1}\nabla\td\Delta_0(\mathbf h;\td c_0)\right\|_2\leq\delta\right\}.
\]
Therefore
\allowdisplaybreaks
\begin{align*}
& \quad \int_{\|z_0\|_2\leq\delta}\exp\left\{-\frac n2\left(z_0-A_{\mathrm{den}}^{-1}\nabla\td\Delta_0(\td c_0)\right)^\top A_{\mathrm{den}}\left(z_0-A_{\mathrm{den}}^{-1}\nabla\td\Delta_0(\td c_0)\right)\right\}\,dz_0\\
& \geq\int_{\|y\|_2\leq\delta/2}\exp\left\{-\frac n2y^\top A_{\mathrm{den}}y\right\}\,dy =_{\xi=\sqrt n\,A_{\mathrm{den}}^{1/2}y} \frac{1}{n\sqrt{\det A_{\mathrm{den}}}}\int_{\left\|A_{\mathrm{den}}^{-1/2}\xi\right\|_2\leq\delta\sqrt n/2}e^{-\|\xi\|_2^2/2}\,d\xi \\
&\geq \frac{1}{n\sqrt{\det A_{\mathrm{den}}}}\int_{\|\xi\|_2\leq \delta \sqrt{n\lambda_{\min}(A_{\mathrm{den}})}/2}e^{-\|\xi\|_2^2/2}\,d\xi =\frac{2\pi}{n\sqrt{\det A_{\mathrm{den}}}}\left(1-\exp\left\{-\frac{n\delta^2\lambda_{\min}(A_{\mathrm{den}})}{8}\right\}\right)
\end{align*}

The last inequality follows since, as
\[
\|\xi\|_2\leq\frac{\delta\sqrt n}{2\|A_{\mathrm{den}}^{-1/2}\|_{\mathrm{op}}}=\frac{\delta}{2}\sqrt{n\lambda_{\min}(A_{\mathrm{den}})},
\]
we have
\[
\left\|A_{\mathrm{den}}^{-1/2}\xi\right\|_2\leq\|A_{\mathrm{den}}^{-1/2}\|_{\mathrm{op}}\|\xi\|_2\leq\frac{\delta\sqrt n}{2}.
\]
Since $A_{\mathrm{num}}\succ0$ and $A_{\mathrm{den}}\succ0$, their inverses are also positive definite. Consequently, for any realization $\mathbf h$, $\Gamma_1(\mathbf h;\td c)\geq0$, and $\Gamma_2(\mathbf h;\td c_0)\geq0$. Therefore, for almost every realization $\mathbf h=\mathbf h^t(\omega)$ and every $n\geq \max\{n_2(\omega),n_4(\omega)\}$, we obtain
\[
\begin{aligned}
\int_{\mathbb R^2}e^{nF(\mathbf h;x,y,0)}\,dx\,dy &\geq
\frac{4\pi}{n\sqrt{\det A_{\mathrm{den}}}}\exp\left\{n\bar F_0(\mathbf h;\td c_0)+\frac n2\Gamma_2(\mathbf h;\td c_0)\right\}\left(1-\exp\left\{-\frac{n\delta^2\lambda_{\min}(A_{\mathrm{den}})}{8}\right\}\right) \\
&\geq  \frac{4\pi}{n\sqrt{\det A_{\mathrm{den}}}}\exp\left\{n\bar F_0(\mathbf h;\td c_0)\right\}\left(1-\exp\left\{-\frac{\delta^2\lambda_{\min}(A_{\mathrm{den}})}{8}\right\}\right).
\end{aligned}
\]
    
    \ppart{Bounding the ratio} Summarizing the results above, we have for almost every realization $\mathbf h=\mathbf h^t(\omega)$ and every $n \geq \max\{n_1(\omega), n_2(\omega), n_3(\omega), n_4(\omega)\}$, 
    \begin{align*}
        &\quad \int_{(x,y,w)\in \R^3}\exp\left\{nF(\mathbf h; x,y,w)\right\}dx\,dy\,dw = 2( \mathsf{I} + \mathsf{II} + \mathsf{III}) \\
        &\leq 2\left(\frac{2\pi}{n}\right)^{3/2}\frac{\exp\left\{n \bar F(\mathbf h; \td c) +\frac{n}{2}\,\Gamma_1(\mathbf h;\td c)\right\}}{\sqrt{\det\left(A_{\text{num}}\right)}} + \frac{8\pi R^3}{3}\exp\left\{n\bar F(\mathbf h; \td c)-\frac{n\gamma_{\delta,R}}{2}\right\} + 8\pi e^{-\gamma_R n} \left(\frac{2R}{n \lambda}+\frac{4}{n^2\lambda^2R}\right) \\
        &= 2\left(\frac{2\pi}{n}\right)^{3/2}\frac{\exp\left\{n \bar F(\mathbf h; \td c) +\frac{n}{2}\,\Gamma_1(\mathbf h;\td c)\right\}}{\sqrt{\det\left(A_{\text{num}}\right)}}\left(1+ \frac{\sqrt{2}R^3 n^{3/2}\exp\left\{-\frac{n}{2}\left(\gamma_{\delta, R}+\Gamma_1(\mathbf h;\td c)\right)\right\}}{3 \sqrt{\pi}} \right. \\
        & \qquad\left.+ \frac{2^{3/2}\sqrt{n} R\exp\left\{-n\left(\gamma_R + \bar F(\mathbf h; \td c) + \frac{1}{2}\Gamma_1(\mathbf h;\td c)\right)\right\}}{\sqrt{\pi}\lambda}+\frac{2^{5/2}\exp\left\{-n\left(\gamma_R + \bar F(\mathbf h; \td c) + \frac{1}{2}\Gamma_1(\mathbf h;\td c)\right)\right\}}{\sqrt{\pi n } \lambda^2 R }\right) \\
        &\leq 2\left(\frac{2\pi}{n}\right)^{3/2}\frac{\exp\left\{n \bar F(\mathbf h; \td c) +\frac{n}{2}\,\Gamma_1(\mathbf h;\td c)\right\}}{\sqrt{\det\left(A_{\text{num}}\right)}}\left(1+ \frac{\sqrt{2}R^3 n^{3/2}e^{-n\gamma_{\delta, R}/2}}{3 \sqrt{\pi}} +  \frac{2^{3/2}\sqrt{n} Re^{-n\gamma_R}}{\sqrt{\pi}\lambda}+\frac{2^{5/2}e^{-n\gamma_R}}{\sqrt{\pi n } \lambda^2 R }\right).
     \end{align*}
    where the last inequality holds since for any realization $\mathbf h$, $\Gamma_1(\mathbf h; \td c)\geq 0$ and
    \allowdisplaybreaks
    \begin{align*}
        \bar F(\mathbf h; \td c) &= -\beta^2\left(m^*\right)^2 + \frac{1}{n} \sum_{i=1}^n \ln(\cosh(2\beta^2m^* + 2h_i)+1)\\
        &=-\beta^2\left(m^*\right)^2 - \ln2 +\frac{2}{n} \sum_{i=1}^n\ln \left(e^{\beta^2 m^*+h_i}+e^{-\beta^2 m^*-h_i}\right) \\
        &\geq -\frac14 -\ln2 +2\ln2 >0.
    \end{align*}
    Note that $\det(A_\text{num})$, $\det(A_\text{den})$, $R$, $\gamma_{\delta, R}$ and $\gamma_R$ are positive constants depending only on $\beta$ and $\lambda$.
    Since the function $x^\alpha e^{-\beta x}$ attains its maximum at $x=\alpha/\beta$, where $\alpha, \beta>0$ and $x \in \mathbb{R}$, we have
    \[
        n^{3/2}e^{-n\gamma_{\delta, R}/2} \leq \sup _{x \geq 0} x^{3 / 2} e^{-\gamma_{\delta, R} x / 2}= \left(\frac{3}{e \gamma_{\delta, R}}\right)^{3 / 2}:=C_{1, 2} , \text{ and }\sqrt{n} e^{-n\gamma_R} \leq \sup _{x \geq 0} x^{1 / 2} e^{-\gamma_R x}=\left(\frac{1}{2 e \gamma_R}\right)^{1 / 2}:= C_{1, 3},
    \]
    where $C_{1, 2}$ and $C_{1, 3}$ are also positive constants depending only on $\beta$ and $\lambda$. Then
    \[
    \begin{aligned}
        \int_{(x,y,w)\in \R^3}\exp\left\{nF(\mathbf h; x,y,w)\right\}dx\,dy\,dw &\leq 2\left(\frac{2\pi}{n}\right)^{3/2}\frac{\exp\left\{n \bar F(\mathbf h; \td c) +\frac{n}{2}\,\Gamma_1(\mathbf h;\td c)\right\}}{\sqrt{\det\left(A_{\text{num}}\right)}} \\
        & \qquad \qquad  \left(1+\frac{\sqrt{2}R^3 C_{1, 2}(\beta, \lambda)}{3\sqrt{\pi}}+\frac{2^{3/2}R C_{1, 3}(\beta, \lambda)}{\sqrt{\pi} \lambda} + \frac{2^{5/2}e^{-\gamma_R}}{\sqrt{\pi}\lambda^2R}\right).
    \end{aligned} 
    \]
    Now, by~\pref{lem:hs-variance-mgf} and the bound for the upper numerator above combined with the lower bound for the denominator, for almost every realization $\mathbf h=\mathbf h^t(\omega)$ and every $n \geq \max\{n_1(\omega), n_2(\omega), n_3(\omega), n_4(\omega)\}$, 
    \begin{align*}
        \an{e^{\lambda n (R_{\sigma,\tau}-q^*)^2}}_\mathbf h &= \sqrt{\frac{\lambda n}{\pi}}\frac{\int_{(x,y,w)\in\R^3}e^{F(\mathbf h; x,y,w)n}dx\,dy\,dw}{\int_{(x,y)\in\R^2}e^{F(\mathbf h;x,y,0)n}dx\,dy} \\
        &\leq \sqrt{\frac{2\lambda\det(A_{\mathrm{den}})}{\det(A_{\text{num}})}}\frac{\left(1+\frac{\sqrt{2}R^3 C_{1,2}}{3\sqrt{\pi}}+\frac{2^{3/2}R C_{1,3}}{\sqrt{\pi} \lambda} + \frac{2^{5/2}e^{-\gamma_R}}{\sqrt{\pi}\lambda^2R}\right)}{\left(1-\exp\left\{-\frac{\delta^2\lambda_{\min}(A_{\mathrm{den}})}{8}\right\}\right)}\exp\left\{n(\bar F(\mathbf h; \td c)-\bar F_0(\mathbf h; \td c_0)) + \frac{n}{2}\Gamma_1(\mathbf h; \td c)\right\} \\
        &=_{\bar F(\mathbf h; \td c) = \bar F_0(\mathbf h; \td c_0)} C_{1,1}(\beta, \lambda) e^{n\Gamma_1(\mathbf h; \td c)/2},
    \end{align*}
    where 
    \[
    C_{1,1}(\beta, \lambda) := \sqrt{\frac{2\lambda\det(A_{\mathrm{den}})}{\det(A_{\text{num}})}}\frac{\left(1+\frac{\sqrt{2}R^3 C_{1,2}}{3\sqrt{\pi}}+\frac{2^{3/2}R C_{1,3}}{\sqrt{\pi} \lambda} + \frac{2^{5/2}e^{-\gamma_R}}{\sqrt{\pi}\lambda^2R}\right)}{\left(1-e^{-\delta^2\lambda_{\min}(A_{\mathrm{den}})/8}\right)} . \qedhere
    \]
\end{prf}

\begin{remark}[Tail-bound version of~\pref{thm:whp-overlap-concentration-laplace-cwrf}]
\label{rmk:tail-bound-version}

We retain all the definitions and notation from~\pref{thm:whp-overlap-concentration-laplace-cwrf}. Define the
good event
\begin{equation}\label{eq:definition-good-event}
    E_n:=E_n^{\mathrm{tail}}\cap E_n^H\cap E_n^{\mathrm{ann}}\cap E_n^A,
\end{equation}
where
\allowdisplaybreaks
\begin{align*}
&E_n^{\mathrm{tail}}:=\left\{\omega:\frac1n\sum_{i=1}^n\left|h_i^t(\omega)\right|\leq\E_{\mathbf h^t}\left[|h_1^t|\right]+1\right\},\\
&E_n^H:=\left\{\omega:\max_{a,b\in\{u,v,w\}}\left|\partial_{a,b}\td\Delta\left(\mathbf h^t(\omega);\td c\right)\right|\leq\frac{\varepsilon}{6}\right\},\\
&E_n^{\mathrm{ann}}:=\left\{\omega:\sup_{z\in B_R(\td c)}\left|\td\Delta\left(\mathbf h^t(\omega);z\right)\right|\leq\frac{\gamma_{\delta,R}}{4}\right\},\\
&E_n^A:=\left\{\omega:\left\|A_{\mathrm{den}}^{-1}\nabla\td\Delta_0\left(\mathbf h^t(\omega);\td c_0\right)\right\|_2\leq\frac{\delta}{2}\right\}.
\end{align*}
Thus, according to~\pref{thm:whp-overlap-concentration-laplace-cwrf}, the good event $E_n$ holds when $n \geq \max\{n_1(\omega), n_2(\omega), n_3(\omega), n_4(\omega)\}$, yielding the bound in~\pref{thm:whp-overlap-concentration-laplace-cwrf} for almost every realization $\mathbf h=\mathbf h^t(\omega)$:
\[
\left\langle e^{\lambda n(R_{\sigma,\tau}-q^*)^2}\right\rangle_{\mathbf h}\leq C_{1,1}(\beta,\lambda)\exp\left\{\frac n2\Gamma_1(\mathbf h;\td c)\right\}.
\]
\end{remark}

We next establish an exponential bound on the probability of the complement of the good event.

\begin{lemma}[Probability of the good event]\label{lem:probability-good-event}
Let $E_n$ be the good event defined in~\pref{eq:definition-good-event}. Then there exist positive constants $C_{1,4}(\beta, \lambda)$ and $C_{1,5}(\beta, \lambda)$, such that
\[
\P(E_n^c)\leq C_{1,4}(\beta, \lambda)e^{-C_{1,5}(\beta, \lambda)n}.
\]
\end{lemma}

\begin{prf}
Since $h_1^t,\ldots,h_n^t \stackrel{\mathrm{i.i.d.}}{\sim} \mathcal N(t,\beta^2 q^* + t)$ and $|\cdot|$ is 1-Lipschitz we can use concentration for Lipschitz functions of sub-Gaussian random variables. Similar to part $\mathrm{II}$, the degenerate case does not change the result, so without loss of generality, we assume $\beta^2 q^* + t \neq 0$. By standard concentration of measure for Lipschitz functions of Gaussians, 
\[ 
\P\left(\left(E_n^{\mathrm{tail}}\right)^c\right) = \P\left(\frac1n\sum_{i=1}^n \left|h_i^t\right| - \E_{\mathbf h^t}|h_1^t|> 1 \right)\leq\exp\left\{-\frac{n}{2(\beta^2 q^* + t)}\right\}. 
\]

Next, we consider $E_n^H$. For any $a,b \in \{u,v,w\}$, previously result gives $\left|\partial_{a,b} \ln A(h_i^t; \td c)\right| \leq 8 \max\{\beta, \lambda\}^2$. In addition, $\partial_{a,b} \ln A(h_i^t; \td c)$ are independent random variables. Hoeffding's inequality therefore yields
\[
\begin{aligned}
    \P\left(\left(E_n^H\right)^c\right) &= \P\left(\max_{a,b\in \{u, v, w\}} \left|\partial_{a,b}\td\Delta\left(\mathbf h^t(\omega);\td c\right)\right|>\frac{\varepsilon}{6}\right) \leq\sum_{a,b\in \{u, v, w\}} \P\left(\left|\partial_{a,b}\td\Delta\left(\mathbf h^t(\omega);\td c\right)\right|\geq\frac{\varepsilon}{6}\right) \\
    &= \sum_{a,b\in \{u, v, w\}} \P\left(\left|\sum_{i=1}^n\partial_{a,b} \ln A(h_i^t; \td c) - \E_{\mathbf h^t}\left[\sum_{i=1}^n \partial_{a,b} \ln A(h_i^t; \td c)\right]\right|\geq\frac{n\varepsilon}{6}\right) \leq 18 \exp\left\{-\frac{n\varepsilon^2}{4608\max\{\beta, \lambda\}^4}\right\}
\end{aligned}
\]

For $E_n^{\mathrm{ann}}$, by $\pref{eq:summation-of-probabilities}$,
\[
\P\left(\left(E_n^{\mathrm{ann}}\right)^c\right)=\P\left(\sup_{z\in B_R(\td c)}\left|\td\Delta(\mathbf h^t;z)\right|>\frac{\gamma_{\delta,R}}{4}\right)\leq2\left(1+\frac{32RL}{\gamma_{\delta,R}}\right)^3\exp\left\{-\frac{n\gamma_{\delta,R}^2}{512(\beta^2 q^* + t)}\right\}.
\]

Finally, we consider $E_n^A$. By our definition above, 
\[
\begin{aligned}
    &\quad \nabla\td\Delta_0\left(\mathbf h^t;\td c_0\right) = \left(\left(\nabla\td\Delta_0\left(\mathbf h^t;\td c_0\right)\right)_u, \left(\nabla\td\Delta_0\left(\mathbf h^t;\td c_0\right)\right)_v\right) \\
    &= \left(\frac1n\sum_{i=1}^n\left(\partial_u\ln A(h_i^t;\td c)-\E_{h^t}\left[\partial_u\ln A(h^t;\td c)\right]\right),\frac1n\sum_{i=1}^n\left(\partial_v\ln A(h_i^t;\td c)-\E_{h^t}\left[\partial_v\ln A(h^t;\td c)\right]\right)\right).
\end{aligned}
\]
In addition, for every $a\in\{u,v\}$, $\left|\partial_a\ln A(h_i^t;\td c)\right|\leq 2\beta$, and $\partial_a\ln A(h_i^t;\td c)$ are independent random variables. Hence, for any $r>0$, Hoeffding's inequality gives
\[
\P\left(\left|\left(\nabla\td\Delta_0(\mathbf h^t;\td c_0)\right)_a\right|>r\right)\leq2\exp\left\{-\frac{nr^2}{8\beta^2}\right\}.
\]
Therefore
\[
\begin{aligned}
\P\left(\left(E_n^A\right)^c\right)&=\P\left(\left\|A_{\mathrm{den}}^{-1}\nabla\td\Delta_0(\mathbf h^t;\td c_0)\right\|_2>\frac{\delta}{2}\right)\leq\P\left(\left\|\nabla\td\Delta_0(\mathbf h^t;\td c_0)\right\|_2>\frac{\delta}{2\left\|A_{\mathrm{den}}^{-1}\right\|_{\mathrm{op}}}\right)\\
&\leq\sum_{a\in\{u,v\}}\P\left(\left|\left(\nabla\td\Delta_0(\mathbf h^t;\td c_0)\right)_a\right|>\frac{\delta}{2\sqrt2\,\left\|A_{\mathrm{den}}^{-1}\right\|_{\mathrm{op}}}\right)\leq 4\exp\left\{-\frac{n\delta^2}{64\beta^2\left\|A_{\mathrm{den}}^{-1}\right\|_{\mathrm{op}}^2}\right\},
\end{aligned}
\]
where in the penultimate inequality, we used the fact that if $\|(u,v)\|_2>r$, then $|x|$ or $|y|$ is greater than $r/\sqrt2$.

Since $E_n^c = \left(E_n^{\mathrm{tail}}\right)^c\cup\left(E_n^H\right)^c\cup\left(E_n^{\mathrm{ann}}\right)^c\cup\left(E_n^A\right)^c$, the union bound gives
\[
\begin{aligned}
    \P(E_n^c)&\leq\exp\left\{-\frac{n}{2(\beta^2 q^* + t)}\right\}+18\exp\left\{-\frac{n\varepsilon^2}{4608\max\{\beta,\lambda\}^4}\right\} \\
    &\qquad +2\left(1+\frac{32RL}{\gamma_{\delta,R}}\right)^3\exp\left\{-\frac{n\gamma_{\delta,R}^2}{512(\beta^2 q^* + t)}\right\}+4\exp\left\{-\frac{n\delta^2}{64\beta^2\left\|A_{\mathrm{den}}^{-1}\right\|_{\mathrm{op}}^2}\right\}.
\end{aligned}
\]
Consequently, we may take
\[
C_{1,4}(\beta, \lambda):=23+2\left(1+\frac{32RL}{\gamma_{\delta,R}}\right)^3
\]
and
\[
C_{1,5}(\beta, \lambda):=\min\left\{\frac{1}{2(\beta^2 q^* + t)},\frac{\varepsilon^2}{4608\max\{\beta,\lambda\}^4},\frac{\gamma_{\delta,R}^2}{512(\beta^2 q^* + t)},\frac{\delta^2}{64\beta^2\left\|A_{\mathrm{den}}^{-1}\right\|_{\mathrm{op}}^2}\right\}.
\]
Then $C_{1,4}(\beta, \lambda), C_{1,5}(\beta, \lambda)\in (0,+\infty)$, and
\[
\P(E_n^c)\leq C_{1,4}(\beta, \lambda)e^{-C_{1,5}(\beta, \lambda)n}. \qedhere
\]
\end{prf}

\paragraph{Boosting the random Laplace approximation to moment estimates} We can use the almost-sure (quenched) upper bound on the MGF of the overlap deviation function from \pref{thm:whp-overlap-concentration-laplace-cwrf} with control over the probability for the ``good event'' established in \pref{lem:probability-good-event} to bound the even moments of the overlap deviation.

\begin{theorem}[Even Moment Bounds for the Overlap Concentration]\label{thm:even-moment-bounds-overlap-concentration}
    For every $k \in \Z_{>0}$, $\beta \in (0, 1/2)$, and some constant $C_1\left(\beta,k\right)$,
    \[
    \underset{\mathbf{h}^t}{\E}\left[\left\langle\left(R_{\sigma, \tau}-q^*\right)^{2 k}\right\rangle_{\mathbf{h}^t}\right] \leqslant \frac{C_1\left(\beta, k\right)}{n^k}.
    \]
\end{theorem}
\begin{prf}
For any $s>0$, similar to deriving the upper bound for the probability of the event $\left(E_n^A\right)^c$ in~\pref{lem:probability-good-event},
\allowdisplaybreaks
\begin{align*}
    \P(n\Gamma_1(\mathbf h^t; \td c) > s) &= \P\left(\left(\sqrt{n}\nabla \td \Delta(\mathbf h^t; \td c)\right)^\top A_{\text{num}}^{-1} \left(\sqrt{n} \nabla \td \Delta(\mathbf h^t; \td c)\right)>s\right) \leq \P\left(\left\|\sqrt{n} \nabla \td \Delta(\mathbf h^t; \td c)\right\|_2^2 > \frac{s}{\left\|A^{-1}_{\text{num}}\right\|_{\text{op}}}\right) \\
    &\leq \sum_{a \in \{u,v,w\}}\P\left(n\left(\nabla \td \Delta(\mathbf h^t; \td c)\right)_a^2 > \frac{s}{3\left\|A^{-1}_{\text{num}}\right\|_{\text{op}}}\right) = \sum_{a \in \{u,v,w\}}\P\left(\sqrt{n}\left|\left(\nabla \td \Delta(\mathbf h^t; \td c)\right)_a\right| > \sqrt{\frac{s}{3\left\|A^{-1}_{\text{num}}\right\|_{\text{op}}}}\right) \\
    & = \sum_{a \in \{u,v,w\}}\P\left(\left|\sum_{i=1}^n \partial_a \ln A(h_i^t; \td c) - \E_{\mathbf h^t} \left[\sum_{i=1}^n \partial_a \ln A(h^t_i;\td c)\right] \right|>\sqrt{\frac{sn}{3\left\|A^{-1}_{\text{num}}\right\|_{\text{op}}}}\right)     \\
    &\leq_{\text{Hoeffding}}3 \cdot 2 \exp\left\{-\frac{2sn}{3\left\|A^{-1}_{\text{num}}\right\|_{\text{op}} n 16 \max\{\beta, \lambda\}^2}\right\} = 6\exp\left\{-\frac{s}{24\max\{\beta, \lambda\}^2\left\|A^{-1}_{\text{num}}\right\|_{\text{op}}}\right\} .
\end{align*}
Fix $\lambda_1$ as any number between 0 and 7/16, for example, $\lambda_1 = 1/10$. Define event $E_s:=\{n\Gamma_1(\mathbf h^t; \td c) \leq \theta s\}$ for some constant $\theta \in (0, 2\lambda_1)$, and let good event be as defined in~\pref{eq:definition-good-event}, then for any $\beta \in (0,1/2)$,
\allowdisplaybreaks
\begin{align*}
\E_{\mathbf h^t} \left[\left\langle \mathbf 1_{\{n\left(R_{\sigma,\tau}-q^*\right)^2>s\}}\right\rangle_{\mathbf h^t}\right]
&= \E_{\mathbf h^t}\left[\mathbf 1_{E_n\cap E_s}\left\langle\mathbf 1_{\{n\left(R_{\sigma,\tau}-q^*\right)^2>s\}}\right\rangle_{\mathbf h^t}\right]+\E_{\mathbf h^t}\left[\mathbf 1_{E_n\cap E_s^c}\left\langle\mathbf 1_{\{n\left(R_{\sigma,\tau}-q^*\right)^2>s\}}\right\rangle_{\mathbf h^t}\right] \\
&\qquad \qquad +\E_{\mathbf h^t}\left[\mathbf 1_{E_n^c}\left\langle\mathbf 1_{\{n\left(R_{\sigma,\tau}-q^*\right)^2>s\}}\right\rangle_{\mathbf h^t}\right]\\
&\leq e^{-\lambda_1 s}\E_{\mathbf h^t}\left[\mathbf 1_{E_s} \mathbf 1_{E_n}\left\langle e^{\lambda_1 n\left(R_{\sigma,\tau}-q^*\right)^2}\right\rangle_{\mathbf h^t}\right]+ \P(E_s^c) + \P(E_n^c) \\
& \leq C_{1,1}(\beta, \lambda_1)e^{-\lambda_1 s}\E_{\mathbf h^t}\left[\mathbf 1_{E_s} e^{n \Gamma_1(\mathbf h^t; \td c)/2}\right] + \P(E_s^c) + \P(E_n^c) \\
& \leq C_{1,1}(\beta, \lambda_1)e^{(\theta/2-\lambda_1) s} + 6e^{-\theta s/\left(24\max\{\beta, \lambda_1\}\left\|A^{-1}_{\text{num}}\right\|_{\text{op}}\right)} + C_{1,4}(\beta, \lambda_1, q^*)e^{-C_{1,5}(\beta, \lambda_1)n} \\
& \leq (C_{1,1}(\beta, \lambda_1) + 6) e^{-C_{1,6}(\beta, \lambda_1)s} + C_{1,4}(\beta, \lambda_1)e^{-C_{1,5}(\beta, \lambda_1)n},
\end{align*}
where 
\[
C_{1,6}(\beta, \lambda_1):= \min\left\{\left(\lambda_1 - \frac{\theta}{2}\right), \frac{\theta}{24\max\{\beta, \lambda_1\}^2\left\|A^{-1}_{\text{num}}\right\|_{\text{op}}}\right\}.
\]
Therefore, by~\pref{eq:expectation-and-integral},
\allowdisplaybreaks
\begin{align*}
    \E_{\mathbf h^t}\left[\left\langle\left(R_{\sigma,\tau}-q^*\right)^{2k}\right\rangle_{\mathbf h^t}\right]&=\frac{k}{n^k}\int_0^\infty s^{k-1}\E_{\mathbf h^t}\left[\left\langle\mathbf 1_{\left\{n\left(R_{\sigma,\tau}-q^*\right)^2>s\right\}}\right\rangle_{\mathbf h^t}\right]\,ds  \\
    &\leq \frac{k}{n^k}\left( (C_{1,1}(\beta, \lambda_1) + 6)\int_0^\infty s^{k-1} e^{-C_{1,6}(\beta, \lambda_1)s} \,ds + C_{1,4}(\beta, \lambda_1)e^{-C_{1,5}(\beta, \lambda_1)n} \int_0^{4n} s^{k-1} \, ds \right) \\
    & =\frac{1}{n^k} \left((C_{1,1}(\beta, \lambda_1) + 6) \frac{k!}{C_{1,6}(\beta, \lambda_1)^k} + C_{1,4}(\beta, \lambda_1)e^{-C_{1,5}(\beta, \lambda_1)n}(4n)^k\right) \\
    &\leq \frac{1}{n^k} \left((C_{1,1}(\beta, \lambda_1) + 6) \frac{k!}{C_{1,6}(\beta, \lambda_1)^k} + C_{1,4}(\beta, \lambda_1)\left(\frac{4k}{eC_{1,5}(\beta, \lambda_1)}\right)^k\right) \\
    &= \frac{C_1(\beta, k)}{n^k},
\end{align*}
where $C_1(\beta, k)$ is defined by
\[
C_1(\beta, k) = (C_{1,1}(\beta, \lambda_1) + 6) \frac{k!}{C_{1,6}(\beta, \lambda_1)^k} + C_{1,4}(\beta, \lambda_1)\left(\frac{4k}{eC_{1,5}(\beta, \lambda_1)}\right)^k,
\]
and the second integral in the first inequality only integrates up to $4n$ is because $n\left(R_{\sigma, \tau}-q^*\right)^2 \leq n\left(1+q^*\right)^2 \leq 4 n$. \qedhere
\end{prf}

Importantly, we do not expect the constant $C_1(\beta,k)$ to be tightened further in a way that drops dependence on $k$, since that would indicate stronger-than-sub-Gaussian tails for the overlap\footnote{\,This would imply that the overlap of two replicas does not behave like a sum of independent random variables at high-temperature, but more like a correlated set of random variables that \emph{actively} conspire to have more cancellations. Note that as $\beta \to 0$ the Gibbs measure should just resemble a random product measure on the hypercube, and so this does not seem plausible.}. The following corollary shows that in general, $R_{\sigma, \tau}$ is not equal to $q^*$, but the asymptotic expectation of $R_{\sigma, \tau}$ is $q^*$. 

\begin{corollary}[Asymptotic expectation]\label{cor:asymptotic-expectation}
$\E_{\mathbf h^t}\left[\an{R_{\sigma, \tau}}_{\mathbf h^t}\right] \neq q^*$ in general. As $n \to \infty$, $\E_{\mathbf h^t}\left[\an{R_{\sigma, \tau}}_{\mathbf h^t}\right] \to q^*$.
\end{corollary}
\begin{prf}
First, we show $\E_{\mathbf h^t}\left[\an{R_{\sigma, \tau}}_{\mathbf h^t}\right] \to q^*$. By~\pref{thm:even-moment-bounds-overlap-concentration}, take $k=1$ gives
\[
\E_{\mathbf h^t}\left[\an{R_{\sigma, \tau}-q^*}_{\mathbf h^t}\right] \leq \left(\E_{\mathbf h^t}\left[\an{\left(R_{\sigma, \tau}-q^*\right)^2}_{\mathbf h^t}\right]\right)^{1/2} \leq \sqrt{\left(\frac{C_{1,1}(\beta, \lambda_1)+6}{C_{1,6}(\beta, \lambda_1)}+\frac{4C_{1,4}(\beta, \lambda_1)}{eC_{1,5}(\beta, \lambda_1)}\right)\frac{1}{n}} \to 0.
\]
Next, to show $\E_{\mathbf h^t}\left[\an{R_{\sigma, \tau}}_{\mathbf h^t}\right] \neq q^*$ in general, let $\alpha=t+\beta^2 q^*$, $n=1$, and $\sigma, \tau \sim \mu_{1, h}$, where
\[
\mu_{1, h}(\sigma) = \frac{e^{\beta^2 \sigma^2 /2 + h\sigma}}{e^{\beta^2 \sigma^2 /2 + h}+e^{\beta^2 \sigma^2 /2 - h}} = \frac{e^{h\sigma}}{2\cosh{h}} .
\]
Since
\[
\an{R_{\sigma, \tau}}_h = \sum_{\sigma, \tau \in \{-1, +1\}} \sigma \tau \frac{e^{h\sigma}}{2\cosh{h}} \frac{e^{h\tau}}{2\cosh{h}} = \frac{\left(e^h\right)^2 + \left(e^{-h}\right)^2 -2}{4 \cosh^2{h}} = \tanh^2{h},
\]
we have
\[
\E_{h^t} \left[\an{R_{\sigma, \tau}}_{h^t}\right]=\E_{g \sim \mathcal N(0,1)}\left[\tanh^2\left(t + \sqrt{\alpha} g\right)\right].
\]
On the other hand, fixed-point equations give
\[
q^* = \E_{g \sim \mathcal N(0,1)}\left[\tanh^2\left(\beta^2 m^* + t + \sqrt{\alpha} g\right)\right].
\]
Define $f(x) =\E_{g \sim \mathcal N(0,1)}\left[\tanh^2\left(x + \sqrt{\alpha} g\right)\right] $, our goal is to show $f^\prime(x)>0$ for $x> 0$, then 
\[
\E_{h^t} \left[\an{R_{\sigma, \tau}}_{h^t}\right] = f(t) \leq f(t+\beta^2 m^*) = q^*.
\]
Note that equality holds only if $m^*=0$. Therefore, when $m^*\neq 0$, we have $\E_{h^t} \left[\an{R_{\sigma, \tau}}_{h^t}\right] \neq q^*$. Using the same method as~\pref{lem:extend-nishimori-condition-beta-1}, pair the positive and negative parts, we have
\allowdisplaybreaks
\begin{align*}
    f^\prime(x) &= \E_{g\sim \mathcal N(0,1)} \left[2 \tanh\left(x+\sqrt{\alpha} g\right) \sech^2\left(x+\sqrt{\alpha} g\right)\right] \\
    &= \E_{r\sim \mathcal N(x,\alpha)} \left[2 \tanh\left(r\right) \sech^2\left(r\right)\right] \\
    &= \frac{2}{\sqrt{2\pi \alpha}}\left(\int_0^{+\infty} \tanh\left(s\right) \sech^2\left(s\right) e^{-(s-x)^2/(2\alpha)} \, ds + \int_{-\infty}^{0} \tanh\left(s\right) \sech^2\left(s\right) e^{-(s+x)^2/(2\alpha)} \, ds \right) \\
    &= \frac{2}{\sqrt{2\pi \alpha}} \int_0^{+\infty} \tanh\left(s\right) \sech^2\left(s\right) \left(e^{-(s-x)^2/(2\alpha)}-e^{-(s+x)^2/(2\alpha)}\right) \, ds >0 .
\end{align*}
The final inequality holds because when $s>0$, $\tanh\left(s\right) \sech^2\left(s\right)>0$ and $e^{-(s-x)^2/(2\alpha)} >e^{-(s+x)^2/(2\alpha)}$.
\end{prf}

The following corollary shows that $\sqrt{n}(R_{\sigma,\tau}-q^*)$ has sub-Gaussian behavior for $0 < \beta < 1/2$.

\begin{corollary}[Sub-Gaussianity for overlap]\label{cor:sub-Gaussianity-for-overlap}
For any $\beta \in (0,1/2)$ and $n \in \Z_{>0}$, $\sqrt{n}(R_{\sigma, \tau}-q^*)$ is sub-Gaussian. Equivalently, there exists some constant $K_1(\beta)<\infty$, such that for any $p,n \in \Z_{>0}$, 
\[
\norm{\sqrt{n}(R_{\sigma, \tau} - q^*)}_{L^p} = \left(\E_{\mathbf h^t}\left[\an{\left|\sqrt{n}(R_{\sigma, \tau} - q^*)\right|^p}_{\mathbf h^t}\right] \right)^{1/p} \leq K_1(\beta) \sqrt{p}.
\]

\end{corollary}
\begin{prf}
By~\pref{thm:even-moment-bounds-overlap-concentration},
\[
\E_{\mathbf h^t}\left[n^k\left\langle\left(R_{\sigma,\tau}-q^*\right)^{2k}\right\rangle_{\mathbf h^t}\right] \leq 2 \max\{C_{1,1}(\beta, \lambda_1)+6, C_{1,4}(\beta, \lambda_1)\} k! \left(\max\left\{\frac{1}{C_{1,6}(\beta, \lambda_1)}, \frac{4}{C_{1,5}(\beta, \lambda_1)}\right\}\right)^k,
\]
where we use $k^k/e^k \leq k!$. When $p=2k$, $k\in \Z_{>0}$, we have
\allowdisplaybreaks
\begin{align*}
    & \quad \left(\E_{\mathbf h^t}\left[\an{\left|\sqrt{n}(R_{\sigma, \tau} - q^*)\right|^p}_{\mathbf h^t}\right] \right)^{1/p} = \left(\E_{\mathbf h^t}\left[\an{n^k\left(R_{\sigma, \tau} - q^*\right)^{2k}}_{\mathbf h^t}\right] \right)^{1/2k} \\
    &\leq \left(2 \max\{C_{1,1}(\beta, \lambda_1)+6, C_{1,4}(\beta, \lambda_1)\}\right)^{1/2k} \sqrt{(k!)^{1/k}} \sqrt{\max\left\{\frac{1}{C_{1,6}(\beta, \lambda_1)}, \frac{4}{C_{1,5}(\beta, \lambda_1)}\right\}} \\
    & \leq \sqrt{2\max\left\{\frac{1}{C_{1,6}(\beta, \lambda_1)}, \frac{4}{C_{1,5}(\beta, \lambda_1)}\right\} \max\{C_{1,1}(\beta, \lambda_1)+6, C_{1,4}(\beta, \lambda_1)\}} \cdot \sqrt{k} \\
    &:= K_1(\beta) \sqrt{k} \leq K_1(\beta) \sqrt{p}
\end{align*}
Note that the probability space is a finite measure space, and in a finite measure space, according to Hölder's inequality, the $L^p$ norm is monotonically increasing with $p$. Therefore as $p=2k-1$,
\[
\left(\E_{\mathbf h^t}\left[\an{\left|\sqrt{n}(R_{\sigma, \tau} - q^*)\right|^p}_{\mathbf h^t}\right] \right)^{1/p} \leq \left(\E_{\mathbf h^t}\left[\an{\left|\sqrt{n}(R_{\sigma, \tau} - q^*)\right|^{2k}}_{\mathbf h^t}\right] \right)^{1/2k} \leq K_1(\beta) \sqrt{k} \leq K_1(\beta) \sqrt{p} . \qedhere
\]
\end{prf}

\paragraph{Bounding the MGF of the squared overlap deviation for \texorpdfstring{$0 < \beta < 1/2$}{0 < beta < 1/2}} We now bootstrap the control over the finite moments to the control the MGF via a Taylor expansion followed by a choice of parameters that leads to a convergent geometric sum.

\begin{theorem}[Exponential moment bound for $0 < \beta < 1/2$]\label{exponential-moment-bound-for-0-beta-1/2}
When $0<\lambda_2 < \min \left\{C_{1,5}\left(\beta, \lambda_1\right)/4, C_{1,6}\left(\beta, \lambda_1\right)\right\}$, for some constant $C_3(\beta, \lambda_2)>0$,
\[
\E_{\mathbf{h}^t}\left[\left\langle e^{\lambda_2 n\left(R_{\sigma, \tau}-q^*\right)^2}\right\rangle_{\mathbf{h}^t}\right] \leq C_3(\beta, \lambda_2).
\]  
\end{theorem}

\begin{prf}
Fix $\beta<1/2$, and recall from~\pref{thm:even-moment-bounds-overlap-concentration}:
\[
    \underset{\mathbf{h}^t}{\E}\left[\left\langle n^k\left(R_{\sigma, \tau}-q^*\right)^{2 k}\right\rangle_{\mathbf{h}^t}\right]\leq \left((C_{1,1}(\beta, \lambda_1) + 6) \frac{k!}{C_{1,6}(\beta, \lambda_1)^k} + C_{1,4}(\beta, \lambda_1)\left(\frac{4k}{eC_{1,5}(\beta, \lambda_1)}\right)^k\right),
\]
then Taylor expansion and Stirling's formula gives
\allowdisplaybreaks
\begin{align*}
    \E_{\mathbf{h}^t}\left[\left\langle e^{\lambda_2 n \left(R_{\sigma, \tau}-q^*\right)^2}\right\rangle_{\mathbf{h}^t}\right]&=\E_{\mathbf{h}^t}\left[\left\langle\sum_{k=0}^{\infty} \frac{\lambda_2^k n^k \left(R_{\sigma, \tau}-q^*\right)^{2 k}}{k!}\right\rangle_{\mathbf{h}^t}\right] \\
    & \leq (C_{1,1}(\beta, \lambda_1) + 6) \sum_{k=0}^{\infty} \left(\frac{\lambda_2}{C_{1,6}(\beta, \lambda_1)}\right)^k + C_{1,4}(\beta, \lambda_1) \sum_{k=0}^{\infty} \left(\frac{4\lambda_2}{eC_{1,5}(\beta, \lambda_1)}\right)^k \frac{k^k}{k!}  \\
    &\leq (C_{1,1}(\beta, \lambda_1) + 6) \sum_{k=0}^{\infty} \left(\frac{\lambda_2}{C_{1,6}(\beta, \lambda_1)}\right)^k + C_{1,4}(\beta, \lambda_1) + \frac{C_{1,4}(\beta, \lambda_1)}{\sqrt{2\pi}} \sum_{k=1}^\infty \left(\frac{4\lambda_2}{C_{1,5}(\beta, \lambda_1)}\right)^k.
\end{align*}
When $\lambda_2 < \min \left\{C_{1,5}\left(\beta, \lambda_1\right)/4, C_{1,6}\left(\beta, \lambda_1\right)\right\}$, these series converges. Hence we can take
\[
C_3(\beta, \lambda_2) := (C_{1,1}(\beta, \lambda_1) + 6) \sum_{k=0}^{\infty} \left(\frac{\lambda_2}{C_{1,6}(\beta, \lambda_1)}\right)^k + C_{1,4}(\beta, \lambda_1) + \frac{C_{1,4}(\beta, \lambda_1)}{\sqrt{2\pi}} \sum_{k=1}^\infty \left(\frac{4\lambda_2}{C_{1,5}(\beta, \lambda_1)}\right)^k < \infty. \qedhere
\]
    
\end{prf}

We can also obtain uniform bounds based on the equivalence relations between sub-Gaussian properties \cite[Proposition 2.6.1]{vershynin2018high}.

\begin{theorem}[Uniform bound of exponential moment for $0 < \beta < 1/2$]\label{thm:uniform-exponential-moment-bound-for-0-beta-1/2}
There exists some $\lambda_2>0$ such that 
\[
\E_{\mathbf h^t}\left[\an{e^{\lambda_2 n(R_{\sigma, \tau}-q^*)^2}}_{\mathbf h^t}\right] \leq 2.
\]
\end{theorem}
\begin{prf}
By~\pref{cor:sub-Gaussianity-for-overlap}, $\sqrt{n}(R_{\sigma, \tau}-q^*)$ is sub-Gaussian. Therefore \cite[Proposition 2.6.1]{vershynin2018high} gives MGF of $\sqrt{n}(R_{\sigma, \tau}-q^*)$. Specifically, there exists $K_3(\beta)>0$ such that
\[
\E_{\mathbf h^t}\left[\an{e^{n(R_{\sigma, \tau}-q^*)^2/K_3(\beta)^2}}_{\mathbf h^t}\right] \leq 2.
\]
Taking $\lambda_2 = 1/K_3(\beta)^2$ yields the result. \qedhere
\end{prf}

\section{Extending overlap concentration to \texorpdfstring{$0 <\beta<1$}{0 <beta < 1} via strong convexity}\label{sec:extending-to-beta-1}

In this section, we extend the result of~\pref{sec:bounding-overlap-moments-for-beta-1/2} to $\beta<1$. We first explain in detail why some of the estimates used in~\pref{sec:bounding-overlap-moments-for-beta-1/2} can not extend to $\beta<1$, and provide counterexamples. We subsequently present a new proof of uniqueness of global maximizer.

\subsection{Counter-example for $1/2 < \beta < 1$}\label{subsec:counterexample} Following the original framework, the argument in~\pref{lem:strict-concavity-f-transformed-basis} encounters an obstruction when we try to extend it to $\beta<1$. If we want it to be true, then for any $u, w \in \R$,
\[
\E_{h^t}\left[\frac{\cosh \left(2 \beta u+2 h^t\right) \cosh (2 \beta v) + e^{-4\lambda w}}{\left(e^{2 \lambda w} \cosh \left(2 \beta u+2 h^t\right)+e^{-2 \lambda w} \cosh (2 \beta v)\right)^2}\right] <1/2.
\]
However, when $v=0$, the expectation above is
\[
\E_{h^t}\left[\frac{\cosh \left(2 \beta u+2 h^t\right) + e^{-4\lambda w}}{\left(e^{2 \lambda w} \cosh \left(2 \beta u+2 h^t\right)+e^{-2 \lambda w}\right)^2}\right] = \E_{h^t}\left[\frac{1}{1+e^{4 \lambda w} \cosh \left(2 \beta u+2 h^t\right)}\right] \to 1,
\]
where the convergence holds by the Dominated Convergence Theorem and the fact that for any $u$ and $\omega$:
\[
\frac{1}{1+e^{4 \lambda w} \cosh(2\beta u +2 h^t(\omega))} \longrightarrow 1.
\]
Therefore, it is impossible to prove that the expectation above is less than $1/2$ for all $(u,v,w)\in\mathbb R^3$. We can also provide a finite counterexample. Let
\[
u=-\frac{t}{\beta}, \quad v=0, \quad w=-\frac{\beta^2 q^* + t +1}{2 \lambda},
\]
in addition,
\[
\E\left[\cosh(2(h^t-t))\right] = \int_{-\infty}^{+\infty} \frac{e^{2(x-t)}+e^{2(t-x)}}{2} \frac{1}{\sqrt{2\pi (\beta^2 q^* + t)}} \exp\left\{-\frac{(x-t)^2}{2(\beta^2 q^* + t)}\right\} \, dx= e^{2(\beta^2 q^* + t)},
\]
Jensen's inequality gives,
\[
\E_{h^t}\left[\frac{1}{1+e^{-4 \lambda \frac{\beta^2 q^* + t +1}{2 \lambda}} \cosh \left(2 (h^t-t)\right)}\right]\geq \frac{1}{1+e^{-2}}> \frac12.
\]
This contradicts the requirement stated above. Therefore, the previous proof is no longer valid. 

\subsection{Uniqueness of $\td c$ as global maximizer for $\beta \in (0,1)$ and some $\lambda(\beta)>0$}
We use a different method to directly prove that $\td c$ is the unique global maximizer of $\td F$ when $\beta < 1$ for some $\lambda(\beta)>0$. We retain all the notation from~\pref{sec:bounding-overlap-moments-for-beta-1/2}.

\begin{theorem}[Uniqueness of the global maximizer for $\beta<1$]
\label{thm:uniqueness-global-maximizer-beta-less-than-one}
Let $\beta\in(0,1)$ and
\[
0<\lambda<\min\left\{\frac14,\frac{1-\beta^2}{16\beta^2}\right\},
\]
then $\td c=(\beta m^*,0,0)$ is the unique global maximizer of $\td F(u,v,w)$.
\end{theorem}

\begin{prf}
Recall the change of basis
\[
\begin{pmatrix}
x\\
y\\
w
\end{pmatrix}=\begin{pmatrix}
1&1&0\\
1&-1&0\\
0&0&1
\end{pmatrix}
\begin{pmatrix}
u\\
v\\
w
\end{pmatrix}=:T\begin{pmatrix}
u\\
v\\
w
\end{pmatrix}.
\]
Since $T$ is invertible, there is a one-to-one correspondence between all $(x, y, w)$ and $(u, v, w)$. In particular, it is equivalent to prove that $\hat c = (\beta m^*,\beta m^*,0)$ is the unique global maximizer of $\hat F(x,y,w)$. 

Using the same proof as in \pref{prop:f-infinite-limits}, we obtain $\lim_{\norm{(x,y,w)}_2 \to \infty} \hat F(x,y,w) = -\infty$, which implies that global maximizers of $\hat F$ exist. Let $(x^*, y^*, w^*)$ denote this maximizer. The idea of the framework in~\pref{sec:bounding-overlap-moments-for-beta-1/2} is trying to prove $-\nabla^2 \hat F(x, y, w) \succ 0$ for all $(x, y, w) \in \mathbb{R}^3$, but the case where $w \to \infty$ presents significant difficulties. To overcome this challenge, we leverage the condition that the gradient of $\hat F$ is $0$ at the critical point to restrict the possible range of $w^*$ from $\mathbb{R}$ to a closed interval, thereby ruling out the case where $w$ is excessively large.

First, for any $h$ and $(x, y, w) \in \mathbb{R}^3$, define spin measure on $\{-1, 1\}^2$:
\[
\pi_{h}(\sigma, \tau) \propto \exp \{(\beta x+h) \sigma+(\beta y+h) \tau+2 \lambda w \sigma \tau\}.
\]
It follows immediately that
\allowdisplaybreaks
\begin{align*}
    \E_{\pi_{h}} \left[\sigma\right] &= \frac{e^{\beta (x + y) +2 h +2\lambda w} + e^{\beta (x-y) -2\lambda w}-e^{-\beta (x + y) -2 h +2\lambda w}-e^{\beta (y-x) -2\lambda w}}{2e^{2\lambda w} \cosh(\beta (x + y) +2h) + 2e^{-2\lambda w} \cosh(\beta (x-y))}\\
    & =\frac{e^{2\lambda w} \sinh(\beta (x + y) +2h) + e^{-2\lambda w} \sinh(\beta(x-y))}{e^{2\lambda w} \cosh(\beta(x + y) +2h) + e^{-2\lambda w} \cosh(\beta(x-y))}, \\
    \E_{\pi_{h}} \left[\tau\right] &= \frac{e^{\beta(x+y) +2 h +2\lambda w} + e^{\beta(y-x) -2\lambda w}-e^{-\beta(x+y) -2 h +2\lambda w}-e^{\beta(x-y) -2\lambda w}}{2e^{2\lambda w} \cosh(\beta(x+y) +2h) + 2e^{-2\lambda w} \cosh(\beta(x-y))} \\
    &=\frac{e^{2\lambda w} \sinh(\beta(x+y) +2h) - e^{-2\lambda w} \sinh(\beta(x-y))}{e^{2\lambda w} \cosh(\beta(x+y) +2h) + e^{-2\lambda w} \cosh(\beta(x-y))}, \\
    \E_{\pi_{h}} \left[\sigma \tau\right] &= \frac{e^{\beta(x+y) +2 h +2\lambda w} + e^{-\beta(x+y) -2 h +2\lambda w} - e^{\beta(x-y) -2\lambda w}-e^{\beta(y-x) -2\lambda w}}{2e^{2\lambda w} \cosh(\beta(x+y) +2h) + 2e^{-2\lambda w} \cosh(\beta(x-y))}\\
    &=\frac{e^{2\lambda w} \cosh(\beta(x+y) +2h) - e^{-2\lambda w} \cosh(\beta(x-y))}{e^{2\lambda w} \cosh(\beta(x+y) +2h) + e^{-2\lambda w} \cosh(\beta(x-y))}.
\end{align*}
Since $\partial_w \hat F(x^*,y^*,w^*)=0$, we have
\[
w^* +q^* = \E_{h^t} \left[\frac{e^{2\lambda w} \cosh\left(\beta(x+y) +2 h^t\right)-e^{-2\lambda w} \cosh\left(\beta(x-y) \right)}{e^{2\lambda w} \cosh\left(\beta(x+y) +2 h^t\right)+e^{-2\lambda w} \cosh\left(\beta(x-y) \right)}\right]= \E_{h^t} \left[\E_{\pi_{h^t}} \left[\sigma \tau\right]\right] \in [-1, 1].
\]
Then $w^* \in [-1-q^*, 1-q^*] \subset [-2, 1]$. Thus, we only need to show for all $(x, y, w) \in \mathbb{R}^2 \times [-1-q^*, 1-q^*]$, $-\nabla^2 \hat{F}(x, y, w) \succ 0$. 

Now, fixing $(x, y) \in \mathbb{R}^2$ and $w \in [-1-q^*, 1-q^*]$, we have
\allowdisplaybreaks
\begin{align*}
    \partial_{x,x} \hat{F}(x, y, w) &= -1 + \beta^2 - \beta^2 \E_{h^t} \left[\left(\frac{e^{2\lambda w} \sinh(\beta(x+y) +2h^t) + e^{-2\lambda w} \sinh(\beta(x-y))}{e^{2\lambda w} \cosh(\beta(x+y) +2h^t) + e^{-2\lambda w} \cosh(\beta(x-y))}\right)^2\right]\\
    &=-1 + \beta^2 \E_{h^t}\left[\Var_{\pi_{h^t}} \left(\sigma\right)\right], \\
    \partial_{y,y} \hat{F}(x, y, w) &= -1 + \beta^2 - \beta^2 \E_{h^t} \left[\left(\frac{e^{2\lambda w} \sinh(\beta(x+y) +2h^t) - e^{-2\lambda w} \sinh(\beta(x-y))}{e^{2\lambda w} \cosh(\beta(x+y) +2h^t) + e^{-2\lambda w} \cosh(\beta(x-y))}\right)^2\right]\\
    &= -1 + \beta^2 \E_{h^t}\left[\Var_{\pi_{h^t}} \left(\tau \right)\right],\\
    \partial_{w, w} \hat F(x,y,w) &= -2\lambda +4\lambda^2 - 4\lambda^2 \E_{h^t}\left[\left(\frac{e^{2\lambda w} \cosh(\beta x +\beta y +2h^t) - e^{-2\lambda w} \cosh(\beta x -\beta y)}{e^{2\lambda w} \cosh(\beta x +\beta y +2h^t) + e^{-2\lambda w} \cosh(\beta x -\beta y)}\right)^2\right]\\
    &= -2\lambda + 4\lambda^2 \E_{h^t} \left[\Var_{\pi_{h^t}} \left(\sigma \tau \right)\right], \\
    \partial_{x,y} \hat{F}(x, y, w) &= \beta^2 \E_{h^t} \left[\frac{e^{8\lambda w} - 1}{\left(e^{4\lambda w} \cosh(\beta x+\beta y +2 h^t) + \cosh(\beta x - \beta y)\right)^2}\right]=\beta^2 \E_{h^t}\left[\Cov_{\pi_{h^t}}(\sigma, \tau)\right], \\
    \partial_{x, w} \hat F(x,y,w)&=2\lambda \beta \E_{h^t}\left[\frac{2\sinh(\beta(x+y) + 2 h^t) \cosh(\beta (x-y)) - 2\cosh(\beta(x+y)+ 2 h^t) \sinh(\beta(x-y))}{\left(e^{2\lambda w} \cosh(\beta(x+y) +2h^t) + e^{-2\lambda w} \cosh(\beta(x-y))\right)^2}\right] \\
    &= 2\lambda \beta\E_{h^t}\left[\Cov_{\pi_{h^t}}(\sigma, \sigma \tau)\right], \\
    \partial_{y, w} \hat F(x,y,w)&= 2\lambda \beta \E_{h^t}\left[\frac{2\sinh(\beta(x+y) + 2 h^t) \cosh(\beta (x-y)) + 2\cosh(\beta(x+y)+ 2 h^t) \sinh(\beta(x-y))}{\left(e^{2\lambda w} \cosh(\beta(x+y) +2h^t) + e^{-2\lambda w} \cosh(\beta(x-y))\right)^2}\right] \\
    &=2\lambda \beta\E_{h^t}\left[\Cov_{\pi_{h^t}}(\tau, \sigma \tau)\right].
\end{align*}
Thus, the Hessian can be written in the following block form
\[
\nabla^2 \hat{F}(x, y, w)=\left(\begin{array}{cc}
A & \mathbf{b} \\
\mathbf{b}^{\top} & d
\end{array}\right),
\]
where 
\[
A=\nabla^2_{x,y} \hat F(x,y,w)= -I_2 + \beta^2 \E_{h^t} \left[\Cov_{\pi_{h^t}}\binom{\sigma}{\tau}\right], \quad \mathbf b = 2 \beta \lambda \E_{h^t}\binom{\Cov_{\pi_{h^t}}(\sigma, \sigma \tau)}{\Cov_{\pi_{h^t}}(\tau, \sigma \tau)}, \quad d = -2\lambda +4\lambda^2 \E_{h^t} \left[\Var_{\pi_{h^t}}(\sigma \tau)\right].
\]

Next, we bound $A, \mathbf b$, and $d$ in the Hessian matrix respectively. For the interaction term $\Cov_{\pi_{h^t}}(\sigma, \tau)$, we have
\[
\begin{aligned}
    \E_{\pi_{h^t}} \left[\tau \mid \sigma \right] &= \sum_{\tau \in \{-1, +1\}} \tau \pi_{h^t}(\tau \mid \sigma) = \sum_{\tau \in \{-1, +1\}} \tau \frac{e^{(\beta y + h^t) \tau +2\lambda w \sigma \tau}}{2\cosh(\beta y + h^t + 2\lambda w \sigma)} = \tanh(\beta y + h^t + 2\lambda w \sigma) \\
    &= \frac{\tanh(\beta y + h^t + 2\lambda w ) + \tanh(\beta y + h^t - 2\lambda w)}{2} + \frac{\tanh(\beta y + h^t + 2\lambda w) - \tanh(\beta y + h^t - 2\lambda w)}{2} \sigma \\
    &:= a^+ + a^- \sigma
\end{aligned}
\]
Then by tower property,
\allowdisplaybreaks
\begin{align*}
    \left|\Cov_{\pi_{h^t}}(\sigma, \tau)\right| &= \left|\E_{\pi_{h^t}} [\sigma \tau] - \E_{\pi_{h^t}}[\sigma] \E_{\pi_{h^t}}[\tau]\right| = \left|\E_{\pi_{h^t}}\left[\sigma \E_{\pi_{h^t}}\left[\tau \mid \sigma \right]\right] - \E_{\pi_{h^t}}[\sigma] \E_{\pi_{h^t}}\left[\E_{\pi_{h^t}}[\tau \mid \sigma]\right]\right| \\
    &= \left|\E_{\pi_{h^t}}\left[\sigma \left(a^+ + a^- \sigma\right)\right] - \E_{\pi_{h^t}}[\sigma] \E_{\pi_{h^t}}\left[a^+ + a^- \sigma\right]\right| = \left|a^+\E_{\pi_{h^t}} \left[\sigma\right] + a^- - a^+ \E_{\pi_{h^t}}\left[\sigma\right] - a^- \left(\E_{\pi_{h^t}} \left[\sigma\right]\right)^2 \right| \\
    &=\left|a^- \Var_{\pi_{h^t}}(\sigma) \right|\leq \left| \frac{\sinh(4\lambda w)}{2\cosh(\beta y + h^t + 2\lambda w)\cosh(\beta y + h^t - 2\lambda w)} \right| = \left| \frac{\sinh(4\lambda w)}{\cosh(2\beta y + 2 h^t)+\cosh(4\lambda w)}\right|\\
    &\leq \frac{\left|\sinh(4\lambda w)\right|}{1+\cosh(4\lambda w)}= \frac{2|\sinh(2\lambda w)|\cosh(2\lambda w)}{2 \cosh^2(2\lambda w)} = \tanh(2\lambda |w|),
\end{align*}
where we use $\Var_{\pi_{h^t}}(\sigma)\leq 1$ and $\tanh r-\tanh s=\sinh (r-s)/(\cosh r \cosh s)$ for any $r,s \in \R$. Then, for any $\mathbf z_0=(z_x,z_y)^\top\in\mathbb R^2$,
\[
\begin{aligned}
    \mathbf z_0^\top \Cov_{\pi_{h^t}} \binom{\sigma}{\tau} \mathbf z_0 &= \Var_{\pi_{h^t}} (\sigma) z_x^2 + \Var_{\pi_{h^t}} (\tau) z_y^2 + 2 \Cov_{\pi_{h^t}} (\sigma, \tau) z_x z_y \leq z_x^2 + z_y^2 + 2 |\Cov_{\pi_{h^t}} (\sigma, \tau)| |z_x z_y| \\
    &\leq (1+|\Cov_{\pi_{h^t}} (\sigma, \tau)|) \|\mathbf z_0\|_2^2 \leq (1+\tanh(2\lambda |w|))\|\mathbf z_0\|_2^2.
\end{aligned}
\]
Divide $\|\mathbf z_0\|_2^2$ to the left and take the supreme value for all $\mathbf z_0$, we have 
\[
\left\| \Cov_{\pi_{h^t}} \binom{\sigma}{\tau}\right\|_{\text{op}} \leq 1 + \tanh(2\lambda |w|) \leq 1 + \tanh(4\lambda).
\]
Therefore,
\[
\Cov_{\pi_{h^t}} \binom{\sigma}{\tau} \preceq (1+\tanh(4\lambda)) I_2 \implies \E_{h^t} \left[\Cov_{\pi_{h^t}} \binom{\sigma}{\tau}\right] \preceq (1+\tanh(4\lambda)) I_2 .
\]
Since $\lambda < (1-\beta^2)/(16\beta^2)$,
\[
-A \succeq\left(1-\beta^2(1+\tanh (4 \lambda))\right) I_2 \succeq \left(1-\beta^2 - 4\beta^2\lambda \right) I_2 \succ \frac{\left(1-\beta^2 \right)}{2} I_2 \succ 0.
\]
Additionally, $\lambda_{\text{min}}(-A)> (1-\beta^2)/2>0$.

For $\mathbf b$, according to Cauchy-Schwarz inequality,
\[
\begin{aligned}
    \|\mathbf{b}\|_2^2&=4 \beta^2 \lambda^2\left(\left(\E_{h^t} \left[\Cov_{\pi_{h^t}}(\sigma, \sigma \tau)\right]\right)^2+\left(\E_{h^t}\left[\Cov_{\pi_{h^t}}(\tau, \sigma \tau)\right]\right)^2\right) \leq 4 \beta^2 \lambda^2 \left(\E_{h^t} \left[\Cov_{\pi_{h^t}}(\sigma, \sigma \tau)^2\right] + \E_{h^t} \left[\Cov_{\pi_{h^t}}(\tau, \sigma \tau)^2\right] \right) \\
    &\leq 4 \beta^2 \lambda^2 \left(\E_{h^t} \left[\Var_{\pi_{h^t}}(\sigma)\Var_{\pi_{h^t}}(\sigma \tau)\right] + \E_{h^t} \left[\Var_{\pi_{h^t}}(\tau)\Var_{\pi_{h^t}}(\sigma \tau)\right]\right) \leq 8\beta^2\lambda^2,
\end{aligned}
\]
The last inequality holds because $\sigma, \tau, \sigma \tau \in \{-1, +1\}$, thus $\Var_{\pi_{h^t}}(\sigma)$, $\Var_{\pi_{h^t}}(\tau)$, and $\Var_{\pi_{h^t}}(\sigma \tau) \leq 1$.

For d, we have,
\[ -d = 2\lambda -4\lambda^2 \E_{h^t}\left[\Var_{\pi_{h^t}}(\sigma \tau)\right] \geq 2\lambda - 4\lambda^2 >\lambda,
\]
where the last inequality holds because $2\lambda-4\lambda^2 - \lambda = \lambda(1-4\lambda)>_{\lambda<1/4} 0$. \\
Finally, for any $\mathbf z = (\mathbf z_0, z_w)^\top \neq \mathbf 0 \in \R^3$,
\allowdisplaybreaks
\begin{align*}
    \mathbf z^{\top}(-\nabla^2 \hat F(x,y,w))\mathbf z &=\mathbf z_0^{\top}(-A) \mathbf z_0 -2 z_w \mathbf{b}^{\top} \mathbf z_0 -d z_w^2 \\
    &= -\left(\mathbf z_0-z_w (-A)^{-1} \mathbf{b}\right)^{\top}A\left(\mathbf z_0-z_w (-A)^{-1} \mathbf{b}\right)-\left(d+\mathbf{b}^{\top} (-A)^{-1} \mathbf{b}\right) z_w^2.
\end{align*}
Since $-A \succ 0$, $z_w$ and $\mathbf z_0-z_w(-A)^{-1} \mathbf b$ can not both be zero (otherwise $\mathbf z_0=(0,0)^\top$ and $z_w=0$, which contradicts the requirement that $\mathbf z \neq \mathbf 0$), the proof is complete provided that $d+\mathbf{b}^{\top} (-A)^{-1} \mathbf{b}<0$. Since $\|\mathbf{b}\|_2^2\leq 8\beta^2 \lambda^2$ and $\lambda_{\text{min}}(-A)>(1-\beta^2)/2$, we have
\[
\mathbf{b}^{\top} (-A)^{-1} \mathbf{b} \leq \lambda_{\text{max}}\left((-A)^{-1}\right) \|\mathbf{b}\|_2^2 = \frac{ \|\mathbf{b}\|_2^2}{\lambda_{\text{min}}(-A)}< \frac{16\beta^2\lambda^2}{1-\beta^2} < \lambda <-d,
\]
which completes the proof.
\end{prf}
The proof of the remaining part is almost identical to the proof in \pref{subsec:laplace-approximation-around-critical-point}. Only two places need to be changed. 
\allowdisplaybreaks
\vspace{-2mm}
\begin{enumerate}
    \item First, we should take $R_1 >\delta$ large enough such that
        \[
        \frac{\lambda R_1^2}{4}-C_0(\beta, \lambda) + \ln2 - \beta^2>0,
        \]
        then $\gamma_{R_1} = \lambda R_1^2/4-C_0(\beta, \lambda)>\beta^2 - \ln2$.
        This implies
        \[
        \gamma_{R_1} + \bar F(\mathbf h; \td c) \geq -\beta^2(m^*)^2 -2\ln 2 + 2\ln 2 + \beta^2\geq 0.
        \]
        This changes all instances of $R$ in the constant $C_{1,1}(\beta, \lambda)$ to $R_1$ and replaces $\frac{2^{5/2}e^{-n\gamma_R}}{\sqrt{\pi n } \lambda^2 R }$ with $\frac{2^{5/2}}{\sqrt{\pi n } \lambda^2 R_1 }$.
        \item Second, in \pref{subsec:laplace-approximation-around-critical-point} we set $\lambda_1=1/10$, whereas here, for any fixed $\beta\in(0,1)$, we set
            \[
            \lambda_3(\beta) = \min \left\{\frac{1}{8}, \frac{1-\beta^2}{32 \beta^2}\right\}.
            \]
\end{enumerate} 
Thus, this implies that
\allowdisplaybreaks
\begin{align*}
    \E_{\mathbf h^t}\left[\left\langle\left(R_{\sigma,\tau}-q^*\right)^{2k}\right\rangle_{\mathbf h^t}\right] &\leq \frac{1}{n^k}\left((C_{2,1}(\beta, \lambda_3(\beta)) + 6) \frac{k!}{C_{1,6}(\beta, \lambda_3(\beta))^k} + C_{1,4}(\beta, \lambda_3(\beta))\left(\frac{4k}{eC_{1,5}(\beta, \lambda_3(\beta))}\right)^k\right)\\
    &:=\frac{C_2(\beta, k)}{n^k},
\end{align*}
where
\[
C_{2,1}(\beta, \lambda) = \sqrt{\frac{2\lambda\det(A_{\mathrm{den}})}{\det(A_{\text{num}})}}\frac{\left(1+\frac{\sqrt{2}R_1^3 C_{1,2}}{3\sqrt{\pi}}+\frac{2^{3/2}R_1 C_{1,3}}{\sqrt{\pi} \lambda} + \frac{2^{5/2}}{\sqrt{\pi}\lambda^2R_1}\right)}{\left(1-e^{-\delta^2\lambda_{\min}(A_{\mathrm{den}})/8}\right)}.
\]

We can immediately extend ~ \pref{cor:sub-Gaussianity-for-overlap} to the case where $\beta < 1$, which shows $\sqrt{n}(R_{\sigma, \tau} -q^*)$ has sub-Gaussian behavior for $1/2<\beta <1$. Let
\[
K_2(\beta):=\sqrt{2\max\left\{\frac{1}{C_{1,6}(\beta, \lambda_3(\beta))}, \frac{4}{C_{1,5}(\beta, \lambda_3(\beta))}\right\} \max\{C_{2,1}(\beta, \lambda_3(\beta))+6, C_{1,4}(\beta, \lambda_3(\beta))\}},
\]
then for any $p,n \in \Z_{>0}$, 
\[
\norm{\sqrt{n}(R_{\sigma, \tau} - q^*)}_{L^p} = \left(\E_{\mathbf h^t}\left[\an{\left|\sqrt{n}(R_{\sigma, \tau} - q^*)\right|^p}_{\mathbf h^t}\right] \right)^{1/p} \leq K_2(\beta) \sqrt{p}.
\]

\subsection{Bounding the MGF of the squared overlap deviation for \texorpdfstring{$0 < \beta < 1$}{0 < beta < 1}}

\begin{theorem}[Exponential moment bound for $0 < \beta < 1$]\label{exponential-moment-bound-for-0-beta-1}
Fix $0<\beta<1$. When $0<\lambda_4 < \min \left\{C_{1,5}\left(\beta, \lambda_3(\beta)\right)/4, C_{1,6}\left(\beta, \lambda_3(\beta)\right)\right\}$, for some constant $C_4(\beta, \lambda_4)>0$,
\[
\E_{\mathbf{h}^t}\left[\left\langle e^{\lambda_4 n\left(R_{\sigma, \tau}-q^*\right)^2}\right\rangle_{\mathbf{h}^t}\right] \leq C_4(\beta, \lambda_4).
\]  
\end{theorem}

\begin{prf}
Taylor expansion and Stirling's formula gives
\allowdisplaybreaks
\begin{align*}
\E_{\mathbf{h}^t}\left[\left\langle e^{\lambda_4 n \left(R_{\sigma, \tau}-q^*\right)^2}\right\rangle_{\mathbf{h}^t}\right]&\leq (C_{2,1}(\beta, \lambda_3(\beta)) + 6) \sum_{k=0}^{\infty} \left(\frac{\lambda_4}{C_{1,6}(\beta, \lambda_3(\beta))}\right)^k + C_{1,4}(\beta, \lambda_3(\beta)) \\
& \qquad + \frac{C_{1,4}(\beta, \lambda_3(\beta))}{\sqrt{2\pi}} \sum_{k=1}^\infty \left(\frac{4\lambda_4}{C_{1,5}(\beta, \lambda_3(\beta))}\right)^k\\
& := C_4(\beta, \lambda_4) <\infty. \qedhere
\end{align*}
    
\end{prf}

Similarly, we can also obtain uniform bounds for $1/2<\beta<1$.

\begin{theorem}[Uniform bound of exponential moment for $0 < \beta < 1$]\label{thm:uniform-exponential-moment-bound-for-0-beta-1}
There exists some $\lambda_4>0$ such that 
\[
\E_{\mathbf h^t}\left[\an{e^{\lambda_4 n(R_{\sigma, \tau}-q^*)^2}}_{\mathbf h^t}\right] \leq 2.
\]
\end{theorem}
\begin{prf}
Since $\sqrt{n}(R_{\sigma, \tau}-q^*)$ is sub-Gaussian as $0<\beta<1$, by \cite[Proposition 2.6.1]{vershynin2018high}, there exists $K_4(\beta)>0$ such that
\[
\E_{\mathbf h^t}\left[\an{e^{n(R_{\sigma, \tau}-q^*)^2/K_4(\beta)^2}}_{\mathbf h^t}\right] \leq 2. 
\]
Taking $\lambda_4 = 1/K_4(\beta)^2$ yields the result. \qedhere
\end{prf}

\newpage
\addtocontents{toc}{\protect\setcounter{tocdepth}{-1}}

\section*{Acknowledgements}
\addtocontents{toc}{\protect\setcounter{tocdepth}{1}}

\iffocs{}{
JS and JSS thank Ewan Davies and Holden Lee for initial discussions. This project grew out of the potential Hessian ascent (PHA) series of papers, when JS and JSS were analyzing various analytic properties about the planted SK model under SL tilt \cite[\S 4]{davies2026potential}.}

\addtocontents{toc}{\protect\setcounter{tocdepth}{-1}}

\renewcommand{\baselinestretch}{0.90}\normalsize
{
    \small\hypersetup{urlcolor=Black}
    \bibliographystyle{alpha_beta_doi}
    \bibliography{doom.bib}
}
\renewcommand{\baselinestretch}{1.0}\normalsize

\addtocontents{toc}{\protect\setcounter{tocdepth}{1}}

\end{document}